\documentclass{article}
\usepackage[utf8]{inputenc}
\usepackage{amsmath,amssymb,amsfonts,bm}
\usepackage{graphicx}
\usepackage{epstopdf}
\usepackage{mathrsfs}
\usepackage{amsthm}
\usepackage{galois,booktabs,ctable,threeparttable,tabularx,algorithm,algorithmic}
\usepackage[misc]{ifsym}
\usepackage{authblk}
\usepackage{subfigure}
\usepackage{geometry,fullpage,cite}
\usepackage[capitalise]{cleveref}
\def\lV{\left\lVert}
\def\rV{\right\lVert}
\def\lv{\left\lvert}
\def\rv{\right\lvert}

\def\l{\langle}
\def\r{\rangle}

\def\N{\mathcal N}
\def\D{\mathcal D}

\def\I{\mathcal I}
\def\E{\mathcal E}

\def\S{\mathcal S}

\def\E{\mathcal E}

\newtheorem{theorem}{Theorem}
\newtheorem{lemma}{Lemma}

\newtheorem{definition}{Definition}
\newtheorem{corollary}{Corollary}
\newtheorem{proposition}{Proposition}

\graphicspath{{figures/}}

\begin{document}
\title{Sample-Efficient Sparse Phase Retrieval via Stochastic Alternating Minimization}

\author[a]{Jian-Feng Cai}
\author[b]{Yuling Jiao}
\author[b]{Xiliang Lu}
\author[a]{Juntao You \thanks{Corresponding author: jyouab@connect.ust.hk}}
\affil[a]{Department of Mathematics\\ The Hong Kong University of Science and Technology\\ Clear Water Bay, Kowloon, Hong Kong SAR, China.}
\affil[b]{School of Mathematics and Statistics, and Hubei Key Laboratory of Computational Science\\ Wuhan University\\ Wuhan 430072, China.}
\maketitle
\begin{abstract}
In this work we propose a nonconvex two-stage \underline{s}tochastic \underline{a}lternating \underline{m}inimizing  (SAM) method for sparse phase retrieval. The proposed algorithm is guaranteed to have an exact recovery from $O(s\log n)$ samples if provided the initial guess is in a local neighbour of the ground truth. Thus, the proposed algorithm is two-stage, first we estimate a desired initial guess (e.g. via a spectral method), and then we introduce a randomized alternating minimization strategy for local refinement. Also, the hard-thresholding pursuit algorithm is employed to solve the sparse constraint least square subproblems. We give the theoretical justifications that SAM find the underlying signal exactly in a finite number of iterations (no more than $O(\log m)$ steps) with high probability. Further, numerical experiments illustrates that SAM requires less measurements than state-of-the-art algorithms for sparse phase retrieval problem.
\end{abstract}
\section{Introduction}
The task of phase retrieval problem is to recover the underlying signal from its magnitude-only measurements. For
simplicity, we consider the real-valued problem, which is to find the target vector $\bm{x}^{\natural}\in \mathbb{R}^n$ from the phaseless system
\begin{align} \label{problem:PR}
y_i=| \langle \bm{a}_i,\bm{x}^{\natural} \rangle | , \quad i=1,2,\cdots,m,
\end{align}
where $\{\bm{a}_i\}_{i=1}^{m}\subset \mathbb{R}^n$ are the sensing vectors, $\{y_i\}_{i=1}^{m}\subset\mathbb{R}_+$ are the observed data, and $m$ is the number of measurements (or the sample size). This problem arises in many fields such as X-ray crystallography \cite{harrison1993phase}, optics \cite{walther1963question}, microscopy \cite{miao2008extending}, and others \cite{fienup1982phase}. Due to the fact that it is easier to record the intensity of the light waves than phase when using optical sensors, the phase retrieval problem is of great importance in the related applications. See \cite{shechtman2015phase} for more detailed discussions about the applications of phase retrieval in engineering.

The phase retrieval problem \eqref{problem:PR} is nonlinear and has different possible solutions. In fact, it can at most recover the underlying signal $\bm{x}^{\natural}$ up to  a sign $\pm1$ (or a global phase $c$ satisfying $|c|=1$ in the complex case). To determine a unique solution (in the sense of $\bm{x}^{\natural}\sim-\bm{x}^{\natural}$) for the phase retrieval problem \eqref{problem:PR}, the system should be overcomplete (i.e., $m>n$) when there is no any priori knowledge about the underlying signal $\bm{x}^{\natural}$. Furthermore, it has been shown that $m=2n-1$ measurements is necessary and sufficient for a unique recovery with generic real sampling vectors \cite{balan2006signal}.
%Additionally, it is interesting to develop practical algorithms for the phase retrieval problem~\eqref{problem:PR}.

In the past decades, a lot of research works have been done to develop practical algorithms for the phase retrieval~\eqref{problem:PR}. It can be traced back to the works of Gerchberg and Saxton \cite{gerchberg1972practical} and Fienup \cite{fienup1982phase} in 1980s. These classical approaches for phase retrieval are mainly based on alternating projections. Though they enjoy good empirical performance and were also widely used \cite{miao1999extending}, they were lack of theoretical guarantees for a long time. On the contrary, recent phase retrieval algorithms, including convex and nonconvex approaches, usually come with theoretical guarantees. Typical convex approaches such as PhaseLift \cite{candes2015phase} and PhaseCut \cite{waldspurger2015phase} linearize the problem by lifting the $n$-dimensional target signal to an $n\times n$ matrix, and thus computationally expensive. Some other convex approaches such as Phasemax \cite{goldstein2016phasemax} and others \cite{hand2018elementary,bahmani2017flexible} do not need to lift the dimension of the signal, but they are not empirically competitive because they depend highly on the so-called anchor vectors that approximate the unknown signal. For nonconvex phase retrieval approaches, the main challenge is how to find a global minimizer and escape from other critical points. To achieve this, some nonconvex algorithms use a carefully designed initial guess that is guaranteed to be close to the global minimizer (the ground truth), and then the estimation is refined to converge to the global minimizer. These approaches often have a provably near optimal sampling complexity $m\sim O(n)$, and they include alternating minimization \cite{netrapalli2013phase}, Wirtinger flow and its variants \cite{candes2015phase1, chen2015solving,zhang2016reshaped,wang2016solving}, Kaczmarz \cite{tan2019phase,wei2015solving}, Riemannian optimization \cite{cai2018solving}, and Gauss-Newton \cite{gao2016gauss,ma2018globally}. Nevertheless, globally convergent first-order methods and greedy methods which contain no designed initialization has been studied in \cite{WaldspurgerPhase,chen2019gradient,tan2019online} recently.  Without a designed initialization, however, the drawback is that more measurements and iterations are normally required. Most recently, the global landscape of nonconvex optimizations are studied in \cite{sun2018geometric,8918236,cai2021solving}, which suggests that there is no spurious local minimum as long as the sample size is sufficiently large; therefore, any algorithm converging to a local minimum finds a global minimum provably.
 %Generally, all the existing provable algorithms need a sampling complexity $m\sim O(n\log^a n)$ with $a\geq 0$, which is (near) optimal.

%\subsection{Sparse Phase Retrieval}
In many applications, despite the fact that the system \eqref{problem:PR} can be well solved if the measurements are overcomplete, one of the most challenging tasks is to recovery the signal with fewer number of measurements. Also, for the large scale problem, the requirement $m>n$ becomes unpractical due to the huge measurements and computation cost. Therefore, lots of attention has been paid to the case of phase retrieval problem when the underlying signal $\bm{x}^{\natural}$ is structured. One common assumption in signal and image processing is that the target signal $\bm{x}^{\natural}$ is usually sparse or approximately sparse (in a transformed domain) in applications related to signal and image processing.  Thus it is possible to determine a unique solution with much fewer measurements when the target signal $\bm{x}^{\natural}$ is known to be sparse. It then comes to the so-called sparse phase retrieval problem.

To be more specific, the sparse phase retrieval problem is to find a sparse signal $\bm{x}^{\natural}\in \mathbb{R}^n$ from the system
\begin{align}\label{problem:sPRori}
y_i=| \langle \bm{a}_i,\bm{x}^{\natural} \rangle |,  \quad i=1,2,\cdots,m, \qquad \text{subject to} \quad \lVert \bm{x}^{\natural}\lVert_0\le s,
\end{align}
where $s$ is the sparsity level of the underlying signal and usually it satisfies $s\ll n$. It shows that with only $m=2s$ measurements, the solution for the problem \eqref{problem:sPRori} can be uniquely determined with real generic measurements \cite{wang2014phase}. Usually, $s$ is small compared to $n$ in the sparse phase retrieval problem, which makes possible that \eqref{problem:sPRori} requires much fewer measurements than $n$ for a successful recovery. Indeed, practical algorithms such as $\ell_{1}$-regularized PhaseLift method \cite{li2013sparse}, sparse AltMin \cite{netrapalli2013phase}, thresholding/projected Wirtinger flow and its variants \cite{cai2016optimal,soltanolkotabi2019structured}, SPARTA \cite{wang2016sparse}, CoPRAM \cite{jagatap2019sample}, and HTP \cite{CAI2022367}, just name a few, can recover the sparse signal successfully from \eqref{problem:sPRori} with high probability when $m \sim O(s^2\log n)$ Gaussian random sensing vectors are used.

%\textbf{Related work}.
Most practical sparse phase retrieval algorithms are extensions of corresponding approaches for the general phase retrieval problem \eqref{problem:PR} to the sparse setting \eqref{problem:sPRori}. Sparse AltMin \cite{netrapalli2013phase} and CoPRAM\cite{jagatap2019sample} extend the popular alternating minimization \cite{gerchberg1972practical,fienup1982phase} for \eqref{problem:PR} with a sparsity constraint. The sparse AltMin estimates alternatively the phase and the non-zero entries of sparse signal with a pre-computed support, and CoPRAM estimates alternatively the phase and the sparse signal. Some other methods, including SPARTA \cite{wang2016sparse} and thresholding/projected Wirtinger flow \cite{cai2016optimal,soltanolkotabi2019structured}, generalize gradient-type algorithms for \eqref{problem:PR} with an extra sparsifying step to find the sparse signal. Recently hard thresholding pursuit (HTP) \cite{CAI2022367} algorithm for the sparse phase retrieval is proposed, which combines the alternating minimization for \eqref{problem:PR} and the HTP algorithm in compressed sensing \cite{foucart2011hard}.
%HTP is computationally more efficient than other sparse phase retrieval algorithms, and it converges in finite steps.

All the aforementioned sparse phase retrieval algorithms also come with a theoretical guarantee. Typically, those algorithms give a successful sparse signal recovery with high probability using only $O(s\log (n/s))$ Gaussian random measurements as long as the initial guess is in a close neighbour the ground truth. Together with a spectral initialization (see \cite{CAI2022367,jagatap2019sample} for instances), the theoretical sampling complexity of those algorithms is $m\sim O(s^2\log n)$.

\paragraph*{Our contributions.}
In this work, we propose a novel stochastic alternating minimization (SAM) algorithm for sparse phase retrieval. The SAM algorithm merges the ideas from alternating minimization, the HTP algorithm, and the random sampling. Theoretically we show that the proposed SAM algorithm converges to the ground truth in no more than $O(\log m)$ iteration. As a comparison, most existing sparse phase retrieval algorithms are proved linearly convergent only. Due to the random sampling technique, the proposed SAM algorithm achieves the best empirical sampling complexity among all existing sparse phase retrieval algorithms. Moreover the SAM algorithm is also computationally more efficient than existing sparse phase retrieval algorithms, which is confirmed by the experimental results.

\paragraph*{Organization.}
The rest of this paper is organized as follows. The notations and problem setting is given in the remaining part of this section. In Section 2 we propose the algorithm, whose theoretical guarantee is given in Section 3. In Section 4 numerical experiments are presented to show the efficiency of the proposed method. The proofs are given in Section 5.

\paragraph*{Notations.}
 For any vector $\bm{x}\in \mathbb{R}^{n}$ and any matrix $\bm{A}\in \mathbb{R}^{m\times n}$, $\bm{x}^{T}$ and $\bm{A}^{T}$ are their transpose respectively. For any $\bm{x}, \bm{y}\in \mathbb{R}^{n}$, we define
 \begin{equation*}
 \bm{x}\odot \bm{y}:=[x_1y_1,~x_2y_2,~\cdots,~x_ny_n]^T
 \end{equation*}
 to be the entrywise product of $\bm{x}$ and $\bm{y}$. For $\bm{x}\in \mathbb{R}^{n}$, $\mathrm{sgn}{\left(\bm{x}\right)}\in\mathbb{R}^{n}$ is defined by $\left[\mathrm{sgn}{\left(\bm{x}\right)}\right]_i=1$ if $x_i\ge0$, and
 $\left[\mathrm{sgn}{\left(\bm{x}\right)}\right]_i=-1$ otherwise. $\lVert \bm{x}\lVert_0$ is the number of nonzero entries of $\bm{x}\in\mathbb{R}^n$, and $\lVert \bm{x}\lVert_2$ is the standard $\ell_2$-norm, i.e. $\lVert \bm{x}\lVert_2=\left(\sum_{i=i}^{n}x_i^2\right)^{1/2}$. For a matrix $\bm{A}\in \mathbb{R}^{m\times n}$, $\lV\bm{A}\rV_2$ denotes its spectral norm. The notation $[n]$ represents $[n]=\left\{1,2,\cdots,n\right\}$. For an index set $\S$, we use $\lv \S\rv$ to denote the cardinality of $\S$. Also, $\bm{x}_{\S}\in\mathbb{R}^{\lv \S\rv}$ (or $[\bm{x}]_{\S}$ sometimes) stands for the vector obtained by keeping only the components of $\bm{x}\in\mathbb{R}^n$ indexed by $\S$, and $\bm{A}_{\S}$ stands for the submatrix of a matrix $\bm{A}\in\mathbb{R}^{m\times n}$ obtained by keeping only the rows indexed by $\S$. For index set $\I$, $\bm{A}(\I,:)$ denotes the submatrix of $\bm{A}$ which keeps rows of $\bm{A}$ indexed by $\I$.  By $O(\cdot)$, we ignore some positive constant. $\lfloor c\rfloor$ is the integer part of the real number $c$. $\mathbb{N}_+$ is the set of positive natural numbers. For $\bm{x}^{\natural}\in \mathbb{R}^{n}$, $x_{\min}^{\natural}$ and $x_{\max}^{\natural}$ are the smallest and largest nonzero components in magnitude of $ \bm{x}^{\natural}$.
To measure the distance of two signals up to a possible sign flip, we define the distance between $\bm{x}$ and $\bm{y}$ as follows:
\begin{equation}\label{eq:dist}
\mathrm{dist}\left(\bm{x},\bm{y}\right)=\min\left\{\lV\bm{x}-\bm{y}\rV_2,\lV\bm{x}+\bm{y} \rV_2 \right\}.
\end{equation}
Throughout the paper, the sensing matrix and the measurement vector are given by
\begin{align} \label{def:Aandy}
\bm{A} = [\bm{a}_1~\bm{a}_2~ \cdots ~\bm{a}_m]^T \in \mathbb{R} ^{m\times n},\ \bm{y} = [y_1~ y_2~\cdots~ y_m]^T \in \mathbb{R} ^m,
\end{align}
where the sensing matrix $\bm{A}\in \mathbb{R}^{m\times n}$ is i.i.d. Gaussian, i.e., the elements of $\bm{A}$ are independently sampled from the standard normal distribution $\N(0,1)$.
%\subsection{Problem Setting}

Then the sparse phase retrieval problem \eqref{problem:sPRori} can be rewritten as to find $\bm{x}^{\natural}$ satisfying
 \begin{align}\label{problem:sPRinmatrix}
 \bm{y}=\lvert \bm{A}\bm{x}^{\natural} \lvert, \quad \text{subject to} \ \lVert \bm{x}^{\natural}\lVert_0 \le s,
 \end{align}
where the sparsity level $s$ is assumed to be known or estimated in advance.

\section{Stochastic Alternative Minimization Algorithms}
In this section, we present the stochastic alternative minimization (SAM) algorithm (\cref{alg:SAM}) to solve the sparse phase retrieval problem. The proposed SAM algorithm is a combination of the random-batch sample selection technique, the alternating minimization described in \cref{subsection21}, and the HTP algorithm \cite{foucart2011hard} presented in \cref{subsection22}.

%To this end, we first introduce the deterministic alternating minimization for sparse phase retrieval in \cref{subsection21}.
%A similar work can be found in \cite{jagatap2019sample} but using a different approach for the subproblem \eqref{subproblem:ori}.
%Then we give a brief introduction to the hard thresholding pursuit (HTP) algorithm, which is used to our solver to the subproblem \eqref{subproblem:ori}, i.e., . The HTP can be viewed as a Newton type algorithm, hence it is very efficient for solving subproblem \eqref{subproblem:ori}. Combining the above ideas with a random-batch selection for the samples, we conclude the stochastic alternative minimization algorithm.

\subsection{Alternating Minimization for Sparse Phase Retrieval}\label{subsection21}

Alternating minimization (or error reduction/alternating projection) algorithms are popular approaches for solving general phase retrieval problem~\eqref{problem:PR} in the applications (e.g., \cite{fienup1982phase,gerchberg1972practical,netrapalli2013phase}). The idea for such algorithms is straightforward --- one simply iterates between the unknown signal and the unknown phases. We rewrite the sparse phase retrieval problem \eqref{problem:sPRinmatrix} as follows:
\begin{align}\label{problem:alm}
\mathop{\mathrm{minimize}}\limits_{ \lV \bm{x}\rV_0\leq s,\; \bm{v}\in\mathbb{V}} \lV \bm{A} \bm{x} - \bm{v}\odot\bm{y}\rV_2^2,
\end{align}
where $\bm{v}$ denotes the phase and $\mathbb{V}=\{-1,1\}^m$ is the space of all possible phases. Then the objective function in \eqref{problem:alm} is minimized alternatively between $\bm{v}$ and $\bm{x}$ in their corresponding constrained sets --- $\bm{v}$ in the phase space $\mathbb{V}$ and $\bm{x}$ in the sparse vector space $\{\bm{x}\in\mathbb{R}^n:\lV \bm{x}\rV_0\leq  s\}$. More precisely, given $\bm{x}_{k}$,
we solve \eqref{problem:alm} by setting $\bm{x}=\bm{x}_{k}$, which gives the phase
\begin{equation}\label{eq:vupdate}
\bm{v}_{k+1}=\mathrm{sgn}(\bm{A}\bm{x}_{k}).
\end{equation}
Then, we fix $\bm{v}=\bm{v}_{k+1}$ and solve \eqref{problem:alm} to obtain $\bm{x}_{k+1}$, i.e.,
%define  and solve update in the signal $\bm{x}$ is obtained by solving the subproblem
\begin{align}\label{subproblem:ori}
\bm{x}_{k+1}=\mathop{\mathrm{arg~min}}\limits_{ \lV \bm{x}\rV_0\leq  s} \lV \bm{A} \bm{x} - \bm{y}_{k+1}\rV_2^2,
\end{align}
where $\bm{y}_{k+1}=\bm{v}_{k+1}\odot\bm{y}$.
The algorithm is summarized in \cref{alg:alm}. With a good initial guess and a proper solver for the subproblem~\eqref{problem:alm}, we can show that \cref{alg:alm} converges at least linearly to the solution, which is given in the following \cref{prop:convergeAlm}.

\begin{proposition}\label{prop:convergeAlm}
Assume that the sensing matrix $\bm{A}\in \mathbb{R}^{m\times n}$ is i.i.d. Gaussian. Let $\{\bm{x}_k\}_{k\geq 0}$ be generated by \cref{alg:alm} with the initial guess $\bm{x}_{0}$ and the subproblem in Step 6 (i.e. \eqref{subproblem:ori}) solved exactly. There exist universal constants $C,C'\geq 0$ and $\zeta\in(0,1)$ such that: whenever $m\ge C s\log (n/s)$ and the initial guess satisfies $\mathrm{dist}\left(\bm{x}_{0},\bm{x}^{\natural}\right)\le \frac{1}{8}\lV \bm{x}^{\natural}\rV_2$, with probability at least $1-e^{-C'm}$, we have
$$
\mathrm{dist}\left(\bm{x}_{k+1},\bm{x}^{\natural}\right)\leq \zeta\cdot\mathrm{dist}\left(\bm{x}_{k},\bm{x}^{\natural}\right),
\quad\forall~k\geq0,
$$
i.e., \cref{alg:alm} converges at least linearly to the solution set.
\end{proposition}
\begin{proof}
The proof of the proposition is given in \cref{subsec:propconvergeAlm}.
\end{proof}

\begin{algorithm}[hbt!]
   \caption{Alternating minimization for sparse phase retrieval}\label{alg:alm}
   \begin{algorithmic}[1]
     \STATE Input: Sensing matrix $\bm{A}\in \mathbb{R}^{m\times n}$, observed data $\bm{y}$, sparsity level $s$, some small stopping criteria parameter $\epsilon$, and maximum iterations $T$.
     \STATE Initialization: Let the initial value $\bm{x}_{0}$ be generated by a given method, e.g., \cref{alg:initialization}.
     \STATE $k=0$
    \REPEAT
    \STATE Compute $\bm{y}_{k+1}= \mathrm{sgn}(\bm{A}\bm{x}_{k})\odot \bm{y}$.
     \STATE Compute $\bm{x}_{k+1}$ by solving
     \begin{align*}%\label{subproblem:alm}
    \bm{x}_{k+1}=\mathop{\mathrm{arg~min}}\limits_{ \lV \bm{x}\rV_0\leq  s}\lV \bm{A} \bm{x} - \bm{y}_{k+1}\rV_2^2.
     \end{align*}
     \STATE $k=k+1$.
     \UNTIL{$\lV\bm{x}_{k+1}-\bm{x}_{k}\rV_2 /\lV\bm{x}_{k}\rV_2\le \epsilon$ or $k\geq T$}.
     \STATE Output $\bm{x}_{k}$.
   \end{algorithmic}
\end{algorithm}

%In \cref{alg:alm}, $K\in \mathbb{N}_+$ is the maximum number of iterations allowed for the algorithm. In fact, We will see in \cref{section:theorems} that \cref{alg:alm} finds the exact underlying signal $\pm \bm{x}^{\natural}$ in finite number of iterations. Therefore, the remaining task in the algorithm is to solve the subproblem~\eqref{subproblem:ori}. The subproblem~\eqref{subproblem:ori} is widely known to be a compressed sensing (CS) problem in the literature. Though nonconvex and NP-hard, there are numerous research in this problem and various algorithms can be applied. In the next subsection, we will recall the hard thresholding pursuit (HTP) algorithm for CS, which gives exact solution to the subproblem efficiently in Gaussian model.

\subsection{Fixed Step HTP for the CS Subproblem}\label{subsection22}

\cref{prop:convergeAlm} presents the local convergence of \cref{alg:alm} if the subproblem \eqref{subproblem:ori} is exactly solved. In practice, we need an inner solver for the subproblem \eqref{subproblem:ori}. We will consider the hard thresholding pursuit (HTP) algorithm to the subproblem \eqref{subproblem:ori} in this subsection. It is worth mentioning that the similar strategy has been used by the CoPRAM introduced in \cite{jagatap2019sample}, where the solver for subproblem \eqref{subproblem:ori} is CoSAMP \cite{NeedellCoSaMP} in compressed sensing.

The subproblem \eqref{subproblem:ori} can be viewed as a compressed sensing (CS) problem. To see this, we  consider a local region near the ground truth, and we notice that  $$\bm{y}_{k+1}=\mathrm{sgn}(\bm{A}\bm{x}_{k})\odot |\bm{A}\bm{x}_{k}|=\bm{A}\bm{x}^{\natural}+\bm{e}_k$$
where $\bm{e}_k=\mathrm{sgn}(\bm{A}\bm{x}_{k})\odot |\bm{A}\bm{x}_{k}|-\bm{A}\bm{x}^{\natural}$, and $\bm{e}_k$ is small as long as $\bm{x}_{k}$ is sufficiently close to $\pm\bm{x}^{\natural}$ in the Gaussian model (see \cref{bound:Ax-y} in \cref{subsec:lemmas}). Therefore, the subproblem \eqref{subproblem:ori} is a typical sparse constrained least squares problem.
%Though nonconvex and NP-hard, there are numerous research on this problem and various algorithms can be applied, and we will use the hard thresholding pursuit (HTP) algorithm in our proposed approach. The HTP was originally introduced in \cite{foucart2011hard} and is adopted for our problem in the next subsection.

Since the sensing matrix $\frac{1}{\sqrt{m}}\bm{A}$ satisfies the restricted isometry property (RIP) with high probability (see \cref{ARIP}), there are a lot of efficient solver to solve \eqref{subproblem:ori}, such as IHT \cite{blumensath2008iterative}, CoSaMP \cite{NeedellCoSaMP}, HTP\cite{foucart2011hard}, PDASC \cite{fan2014primal}, and many others. For completeness, we give the definition of RIP as follows.

%Let the subproblem \eqref{subproblem:ori} be rewritten into the following general form
%\begin{align}\label{problem:CS}
%     \mathop{\mathrm{minimize}}\limits_{ \lV \bm{x}\rV_0\leq  s} \lV \bm{C} \bm{x} - \bm{d}\rV_2^2,
%\end{align}
%where $\bm{C}\in\mathbb{R}^{m\times n}$ is the sensing matrix, $\bm{d}\in\mathbb{R}^m$ is the (noisy) measurements given by $\bm{d}=\bm{C} \bm{x}^{\natural}+\bm{e}$, $\bm{e}\in \mathbb{R}^n$ is the noise. To ensure a successful reconstruction in $\bm{x}$ in \eqref{problem:CS}, one often assumes the sensing matrix $\bm{C}\in\mathbb{R}^{m\times n}$ satisfies the restricted isometry property (RIP). The RIP was originally presented in \cite{candes2005decoding} and is defined as follows.
\begin{definition}[{\cite{candes2005decoding}}]\label{def:RIP}
A matrix $\bm{C}\in\mathbb{R}^{m\times n}$ satisfies the restricted isometry property (RIP) of order $r$ if there exists $\delta\in [0,1)$ such that
\begin{equation}\label{eq:RIP}
\left(1-\delta\right)\lV\bm{x}\rV_2^2\le\lV \bm{C}\bm{x}\rV_2^2\le \left(1+\delta\right)\lV\bm{x}\rV_2^2, \qquad\forall~\bm{x}~:~\lV\bm{x}\rV_0\le r.
\end{equation}
The $r$-RIP constant (RIP constant of order $r$) $\delta_s$ is defined to be the smallest $\delta$ such that \eqref{eq:RIP} holds.
\end{definition}

%\noindent When $\bm{C}$ satisfies the RIP, a lot of efficient solvers are available for \eqref{problem:CS}, such as IHT \cite{blumensath2008iterative}, CoSaMP \cite{NeedellCoSaMP}, HTP\cite{foucart2011hard}, PDASC \cite{fan2014primal}, and others.
%%After a proper re-scale the sensing matrix $\bm{A}$, a lot of efficient solvers for the problem~\eqref{problem:CS} can be directly applied, e.g.,  IHT \cite{blumensath2008iterative}, CoSaMP \cite{NeedellCoSaMP}, HTP\cite{foucart2011hard}, PDASC \cite{fan2014primal} and others.
%As we consider the standard Gaussian model for problem \eqref{subproblem:ori} , then the matrix $\frac{1}{\sqrt{m}}\bm{A}$ satisfies RIP condition \eqref{eq:RIP} with high probability (see \cref{ARIP}). In this case, the aforementioned algorithms can be applied for solving \eqref{subproblem:ori} efficiently. Indeed, CoSAMP \cite{NeedellCoSaMP} is used in COPRAM \cite{jagatap2019sample}.

In this work, we consider the hard thresholding pursuit (HTP) algorithm with a fixed number of iterations $L$ (e.g., $L=1$) for the subproblem. To make the paper self-contained, we give the HTP algorithm for solving the subproblem~\eqref{subproblem:ori} in \cref{alg:htp}.
%The major computational cost in \cref{alg:htp} is the update
%$$
%\bm{x}_{k,\ell}\leftarrow \mathop{\mathrm{arg~min}}\limits_{\mathrm{supp}(\bm{x})\subset S_{k,\ell}}\lVert\bm{A} \bm{x}-\bm{y}_{k+1}\lVert_{2}.
%$$
%It is equivalent to solving the normal equation
%\begin{equation*}
%\bm{A}^{T}_{S_{k,\ell}}\bm{A}_{S_{k,\ell}}[\bm{x}_{k,\ell}]_{S_{k,\ell}}=\bm{A}^{T}_{S_{k,\ell}}\big(\mathrm{sgn}\left(\bm{A}\bm{x}_{k}\right)\odot\bm{y}\big),
%\end{equation*}
%and $\bm{x}_{k,\ell}$ is supported on $S_{k,\ell}$. As $\lv S_{k,\ell}\rv=s$ and $s\ll n$, the above normal equation can be solved in just $O(s^2 m)$ operations. Thus the computational cost of HTP described in \cref{alg:htp} is $O(mn+s^2m)$ per iteration. Therefore, in case of $s\lesssim \sqrt{n}$, the cost of HTP (\cref{alg:htp}) per iteration is the same order as $O(mn)$. Since HTP can be interpreted as a second-order method, it converges very quickly.
Line 6 in \cref{alg:htp} is equivalent to solving the normal equation
\begin{equation*}
\bm{A}^{T}_{S_{k,\ell}}\bm{A}_{S_{k,\ell}}[\bm{x}_{k,\ell}]_{S_{k,\ell}}=\bm{A}^{T}_{S_{k,\ell}}\bm{y}_{k+1},
\end{equation*}
which can be solved in just $O(s^2 m)$ operations provided $s\lesssim m$. With a fixed number of iterations, the whole computational cost for \cref{alg:htp} is $O(s^2m + mn)$.

Consider $\bm{y}_{k+1}=\bm{A}\bm{x}^{\natural}+\bm{e}_k$. According to \cite[Theorem 5]{bouchot2016hard}, for $\lV\bm{e}_k\rV_2=0$, if $\frac{1}{\sqrt{m}}\bm{A}$ satisfies RIP with constant $\delta_{3s}\leq\frac13$ , then HTP requires no more than $2s$ iterations to give an exact solution of \eqref{subproblem:ori}. For $\lV\bm{e}_k\rV_2$ sufficiently small, \cite[Theorem 6]{bouchot2016hard} further ensures that HTP gives a solution proportional to $\lV\bm{e}_k\rV_2$ in at most $3s$ iterations. Therefore, if HTP is stopped in $O(s)$ iterations, the total computational cost of HTP solving subproblem~\eqref{subproblem:ori} is no more than $O(smn)$ in case of $s\lesssim \sqrt{n}$.

\begin{algorithm}[hbt!]
   \caption{Hard Thresholding Pursuit (HTP) for solving \eqref{subproblem:ori}}\label{alg:htp}
   \begin{algorithmic}[1]
   \STATE Input: Sensing matrix $\bm{A}\in \mathbb{R}^{m\times n}$, data $\bm{y}_{k+1}$, sparsity level $s$, maximum allowed inner iterations $L$.
     \STATE Initialization: $\ell=0$, $\bm{x}_{k,0}=\bm{0}$ (or $\bm{x}_{k,0}=\bm{x}_{k}$).
     \FOR {$\ell=1,2,...,L$}
     \STATE $S_{k,\ell}\leftarrow$ \{indices of $s$ largest entries in magnitude of $\bm{x}_{k,\ell-1}+\frac{1}{m} \bm{A}^T( \bm{y}_{k+1}-\bm{A} \bm{x}_{k,\ell-1})$\}.
     \STATE Solve $\bm{x}_{k,\ell}\leftarrow \mathop{\mathrm{arg~min}}\limits_{\mathrm{supp}(\bm{x})\subset S_{k,\ell}}\lVert\bm{A} \bm{x}-\bm{y}_{k+1}\lVert_{2}$.
     \ENDFOR
     \STATE  Output $\bm{x}_{k+1}=\bm{x}_{k,L}$.
     \end{algorithmic}
\end{algorithm}
%Therefore, the HTP is set to stop in $L$ iterations when solving \eqref{subproblem:alm}.

\subsection{Stochastic Alternating Minimization for Sparse Phase Retrieval}
The proposed stochastic alternating minimization (SAM) algorithm is a stochastic version of a combination of \cref{alg:alm} and \cref{alg:htp}. Stochasticity is also an important ingredient in many algorithms for the standard phase retrieval \eqref{problem:PR}, such as SGD \cite{tan2019online}, Kaczmarz \cite{tan2019phase,wei2015solving}, and STAF \cite{wang2017scalable}. Besides, it has been shown that gradient-based algorithms enjoy better performance when some measurements are ruled out \cite{wang2016solving,wang2017scalable,wang2016sparse,wang2017solving}. Here we adopt stochasticity for solving the sparse phase retrieval to obtain our SAM algorithm.

In each iteration of our proposed SAM algorithm, we first choose samples from the measurement mddel \eqref{problem:sPRori} in a random batch manner, and then we apply one step of alternating minimization algorithm (\cref{alg:alm}) to this random measurement and \cref{alg:htp} is applied to solve the CS subproblem (Line 6 in \cref{alg:alm}).

To describe the full algorithm, we first give the following definition of the Bernoulli sampling model.
\begin{definition}\label{def:samplingmodel}
Let the index set $\I$ be a subset of $[m]$. We say $\I$ is Bernoulli sampled from $[m]$ with a probability parameter $\beta\in (0,1]$ if each entry $i$ is sampled (or kept)  with probability $\beta$ independent of all others.
\end{definition}

Practically, the parameter $\beta$ is almost identical to its empirical version $\lvert\I \rvert /m$. Then we explain one step iteration of the proposed SAM algorithm. Let $\bm{x}_k$ be the estimation at the $k$-th iteration. To start the $(k+1)$-st iteration, we first randomly draw a subset $\I_{k+1}$ of $[m]$ using the Bernoulli sampling model according to \cref{def:samplingmodel}. By keeping the rows of $\bm{A}$ indexed by $\I_{k+1}$, we obtain a submatrix $\bm{A}_{k+1}:=\bm{A}(\I_{k+1},:)$ of $\bm{A}$ at the $(k+1)$-st iteration. We then apply one step iteration of the alternating minimization method to the problem
\begin{equation}\label{eq:}
\mathop{\mathrm{minimize}}\limits_{ \lV \bm{x}\rV_0\leq s,\; \bm{v}\in\mathbb{V}} \lV \bm{A}_{k+1} \bm{x} - \bm{v}\odot\bm{y}_{\I_{k+1}}\rV_2^2.
\end{equation}
This leads to an algorithm where $\bm{A}$ and $\bm{y}$ in \cref{alg:alm} are replaced by $\bm{A}_{k+1}$ and $\bm{y}_{\I_{k+1}}$ respectively. In particular, similar to \eqref{eq:vupdate} and \eqref{subproblem:ori}, we update
%let $\mathrm{sgn}(\bm{A}_{k+1} \bm{x}_{k})$ be the guess of the phase information, we update
$$\bm{y}_{k+1}=\mathrm{sgn}(\bm{A}_{k+1} \bm{x}_{k})\odot \bm{y}_{\I_{k+1}},$$
and
%where $\bm{y}_{\I_{k+1}}:=\bm{y}\left(\I_{k+1}\right)$. Then the signal $\bm{x}$ is updated by solving the following CS-subproblem by HTP
\begin{align*}
\bm{x}_{k+1}\leftarrow \mathop{\mathrm{arg~min}}\limits_{ \lV \bm{x}\rV_0\leq s} \lV \bm{A}_{k+1} \bm{x} - \bm{y}_{k+1}\rV_2^2,
\end{align*}
where the latter is again solved by $L$ steps of HTP algorithm (\cref{alg:htp}). The full algorithm is summarized in \cref{alg:SAM}.

%There are more details contained in \cref{alg:SAM}. for some given constant $c\in(0,1]$, $\beta$ can be any fixed positive constant which actually determines a probability
%\begin{align}\label{def-p}
%p=\frac{\lfloor \beta m\rfloor}{m},~\beta \in [c,1],
%\end{align}
%which is the probability such that $i\in \I_{k+1}$ for any $i\in [m]$. When using HTP to solve \eqref{subproblem:sto}, one thing we shall notice is that the step size $1/m$ and sensing matrix $\bm{A}$ should be replaced by $\frac{1}{\lfloor \beta m\rfloor}$ and $\bm{A}_{k+1}$ respectively in Algorithm 2. Thus we need to iterate the following scheme to solve subproblem \eqref{subproblem:sto}: start with $\bm{x}_{k,0}=\bm{x}_{k}$ (or $\bm{x}_{k,0}=0$),
%\begin{enumerate}\label{htp:sto}
%  \item[(HTP-1)] $S_{k,\ell}\leftarrow$ \{indices of $s$ largest entries in magnitude of $\bm{x}_{k,\ell-1}+\frac{1}{\lfloor \beta m\rfloor}[\bm{A}_{k+1}]^T( \bm{y}_{k+1}-\bm{A}_{k+1} \bm{x}_{k,\ell-1})$\},
%  \item[(HTP-2)]$\bm{x}_{k,\ell}\leftarrow \mathop{\mathrm{arg~min}}\limits_{\mathrm{supp}(\bm{x})\subset S_{k,\ell}}\lVert\bm{A}_{k+1} \bm{x}-\bm{y}_{k+1}\lVert_{2}$,
%\end{enumerate}
%and output $\bm{x}_{k,\ell}$ until some stopping criterion is met.

\begin{algorithm}[hbt!]
   \caption{Stochastic alternating minimization (SAM) for sparse phase retrieval}\label{alg:SAM}
   \begin{algorithmic}[1]
     \STATE Input: the sensing matrix $\bm{A}\in \mathbb{R}^{m\times n}$, the observed data $\bm{y}$, the sparsity level $s$, a fixed constant $\beta\in (0,1]$, a small stopping criteria parameter $\epsilon$, and the maximum number $K$ and $L$ of iterations allowed.
     \STATE Initialization: Let the initial value $\bm{x}_{0}$ be generated by a given method, e.g., \cref{alg:initialization}.
     \STATE $k=0$
     \REPEAT
    \STATE Draw a set $\I_{k+1}$ randomly from $\{1,2,\cdots,m\}$ by Bernoulli sampling model with probability $\beta$ (according to \cref{def:samplingmodel}). Let $\bm{A}_{k+1} = \bm{A}(\I_{k+1},:)$, and $\bm{y}_{\I_{k+1}}:=\bm{y}\left(\I_{k+1}\right)$.
     \STATE Compute
     $$
     \bm{y}_{k+1}= \mathrm{sgn}(\bm{A}_{k+1}\bm{x}_{k})\odot \bm{y}_{\I_{k+1}}.
     $$
     \STATE Obtain $\bm{x}_{k+1}$ by solving
     \begin{align}\label{subproblem:sto}
     \mathop{\mathrm{minimize}}\limits_{ \lV \bm{x}\rV_0\leq  s} \frac{1}{2\beta m}\lV \bm{A}_{k+1} \bm{x} - \bm{y}_{k+1}\rV_2^2
     \end{align}
     via $L$ steps of HTP as in the following: starting from $\bm{x}_{k,0}=\bm{x}_{k}$ (or $\bm{x}_{k,0}=\bm{0}$)
     \begin{itemize}
     \item[] {\bf for} $\ell=1,2,...,L$
     \item[] \qquad\qquad$S_{k,\ell}\leftarrow$ indices of $s$ largest entries in magnitude of $$\bm{x}_{k,\ell-1}+\frac{1}{\beta m}\bm{A}_{k+1}^T( \bm{y}_{k+1}-\bm{A}_{k+1} \bm{x}_{k,\ell-1}).$$
     \item[] \qquad\qquad $\bm{x}_{k,\ell}\leftarrow \mathop{\mathrm{arg~min}}\limits_{\mathrm{supp}(\bm{x})\subset S_{k,\ell}}\lVert\bm{A}_{k+1} \bm{x}-\bm{y}_{k+1}\lVert_{2}$
     \item[] {\bf end for}
     \item[] Set $\bm{x}_{k+1}=\bm{x}_{k,L}$.
     \end{itemize}
     \STATE $k=k+1$
     \UNTIL{$\lV\bm{x}_{k+1}-\bm{x}_{k}\rV_2 /\lV\bm{x}_{k}\rV_2\le \epsilon$ or $k\geq K$.}
     \STATE Output $\bm{x}_{k+1}$.
   \end{algorithmic}
\end{algorithm}

\section{Theory}\label{section:theorems}
In this section, we present some theoretical results of the proposed SAM algorithm (\cref{alg:SAM}). The proofs of these results are postponed to \cref{section:proofs}. Our results show that, despite of the nonconvexity, \cref{alg:SAM} is guaranteed to converge to the underlying sparse signal $\bm{x}^{\sharp}$ in at most $O(\log m)$ steps. As a comparison, most existing nonconvex sparse phase retrieval algorithms (except for HTP\cite{CAI2022367}) are usually guaranteed to have a linear convergence only. Our previous work HTP \cite{CAI2022367} is guaranteed to have finite step convergence, but the number of iterations depends on the dynamics of the underlying signal.

The proof strategy is the same as many popular nonconvex (sparse) phase retrieval solvers. We first show that \cref{alg:SAM} converges to $\bm{x}^{\sharp}$ if the initial guess $\bm{x}_{0}$ is sufficiently close to $\bm{x}^{\sharp}$. Then, we initialize \cref{alg:SAM} by a spectral initialization, which gives $\bm{x}_{0}$ in the basin of local convergence of the algorithm.

\subsection{Local Convergence}\label{subsection:local}
In this subsection, we present local convergence results of \cref{alg:SAM}, under the assumption that the sensing matrix $\bm{A}$ is a standard Gaussian random matrix.

Our result show that, in a $O(\|\bm{x}^{\natural}\|_2)$-neighbourhood of $\pm \bm{x}^{\natural}$, \cref{alg:SAM} with $m=O(s\log(n/s))$ finds the exact solution $\pm \bm{x}^{\natural}$ with high probability. If the subproblem~\eqref{subproblem:sto} is solved exactly by HTP (i.e., $L=2s$), \cref{alg:SAM} terminates in finite number (no more than $O(\log m)$) of iterations. If the subproblem~\eqref{subproblem:sto} is only solved approximately after a fixed number of iterations of HTP (e.g., $L=1$), \cref{alg:SAM} converges linearly. The result is summarized in the following theorem, whose proof is given in \cref{subsec:localconvergence}.

\begin{theorem}[Local Convergence of SAM]\label{localconvergence}
Let $\bm{x}^{\natural}\in\mathbb{R}^{n}$ be satisfying $\|\bm{x}^{\natural}\|_0\leq s$, and $\bm{A}\in\mathbb{R}^{m\times n}$ be a random matrix whose entries are i.i.d. standard normal random variables. Let $\{\bm{x}_{k}\}_{k\ge 0}$ be generated by \cref{alg:SAM} with $\bm{y}:=\lv \bm{A}\bm{x}^{\natural}\rv$, $\beta\in [\frac{1}{10},1]$, $L\geq 1$, and an initial guess $\bm{x}_{0}$ satisfying
$$
\mathrm{dist}\left(\bm{x}_{0},\bm{x}^{\natural}\right)\le \frac{\sqrt{\beta}}{8}\lV \bm{x}^{\natural}\rV_2 .
$$
There exists universal constants $C_1, C_2, C_3>0$ and $\alpha_0\in(0,1)$ such that: whenever $m\ge C_2 \beta^{-2}s\log(n/s)$,
\begin{enumerate}
\item[(a)] with probability at least $1-2Ke^{-C_3 \beta^2m}$,
\begin{align*}
\mathrm{dist}\left(\bm{x}_{k+1},\bm{x}^{\natural}\right)\le \alpha_0\cdot \mathrm{dist}\left(\bm{x}_{k},\bm{x}^{\natural}\right),~\forall~k=0,1,2,\cdots,K-1.
\end{align*}
\item[(b)] if $L\ge 2s$, then with probability at least $1-2Ke^{-C_3 \beta^2m}-m^{-1}$,
\begin{align*}
\mathrm{dist}\left(\bm{x}_{k+1},\bm{x}^{\natural}\right)
\begin{cases}
\le \alpha_0\cdot \mathrm{dist}\left(\bm{x}_{k},\bm{x}^{\natural}\right), &\forall~k=0,1,2,\cdots,K-1,\cr
=0, &\forall~k\ge K-1.
\end{cases}
\end{align*}
for some $K\leq C_1\log m$.
\end{enumerate}
\end{theorem}

From \cref{localconvergence}, we see that, with a good initialization, the computational cost of \cref{alg:SAM} with $L=2s$ is $O(smn\log m)$ for an exact recovery. The lower bound of $\beta\in[\frac{1}{10},1]$ in \cref{localconvergence} is chosen arbitrarily, we can change it to any number smaller in $(0,1)$. In this case, one should change all the constants accordingly.  In our experiments, we simply choose $\beta\ge 0.4$ for the purpose of performance comparison of the algorithms.

We consider an interesting special case of \cref{alg:SAM} and \cref{localconvergence} where $\beta=1$ and $L= 2s$. Since the Bernoulli probability $\beta=1$, there is no stochasticity. Therefore, in this special case, \cref{alg:SAM} with $\beta=1$ is an implementation of \cref{alg:alm}, where the subproblem in Line 6 is solved by HTP. As a by-product, our \cref{localconvergence} gives the following \cref{converge:alm}, which states that the alternating minimization for sparse phase retrieval with a CS subproblem solver enjoys a finite step convergence. This result improves the corresponding result in, e.g., \cite{jagatap2019sample}, where only linear convergence of \cref{alg:alm}. The proof is presented in \cref{subsec:proofalm}.

\begin{corollary}\label{converge:alm}
Let $\bm{x}^{\natural}\in\mathbb{R}^{n}$ be satisfying $\|\bm{x}^{\natural}\|_0\leq s$, and $\bm{A}\in\mathbb{R}^{m\times n}$ be a random matrix whose entries are i.i.d. standard normal random variables. Let the iteration sequence $\{\bm{x}_{k}\}_{k\ge 0}$ generated by \cref{alg:alm} with $\bm{y}:=\lv \bm{A}\bm{x}^{\natural}\rv$, \eqref{subproblem:ori} solved via $L=2s$ iterations of HTP, and the initial guess $\bm{x}_0$ satisfying
$$
\mathrm{dist}\left(\bm{x}_{0},\bm{x}^{\natural}\right)\le \frac{1}{8}\lV \bm{x}^{\natural}\rV_2.
$$
Then there exists universal constants $C_4, C_5, C_6>0$ and $\alpha_1\in(0,1)$ such that: if $m\ge C_5 s\log(n/s)$, then with probability at least $1-e^{-C_6 m}-m^{-1}$, we have
\begin{equation*}
\mathrm{dist}\left(\bm{x}_{k+1},\bm{x}^{\natural}\right) \left\{
\begin{aligned}
&\le \alpha_1 \cdot \mathrm{dist}\left(\bm{x}_{k},\bm{x}^{\natural}\right), \quad &\forall~ k&=0,1,\cdots,T-1, \\
&=0,   &\forall~k&\ge T-1.
\end{aligned}
\right.
\end{equation*}
for some $T\leq C_4\log m$.
\end{corollary}

Notice that though the corollary is applied to \cref{alg:alm} with the HTP solver as in \cref{alg:SAM}, it is easy to extend the corollary to other robust CS solvers like IHT, CoSaMP and PDASC as mentioned in \cref{subsection22}. Since CoPRAM \cite{jagatap2019sample} is actually \cref{alg:alm} where the subproblem in Line 6 is solved exactly by CoSAMP, \cref{converge:alm} yileds that CoPRAM gives the exact sparse phase retrieval in at most $O(\log m)$ steps. This new result is better than the theoretical result in the original paper \cite{jagatap2019sample} of CoPRAM, where only linear convergence rate has been proved.

\subsection{Initialization and Exact Recovery}\label{subsection:global}

As problem~\eqref{problem:alm} is non-convex, a good initialization is of great importance for the convergence of the algorithm. For problems like the standard phase retrieval with gaussian sensing matrix, random initial guess has been proved to work but at a cost of more measurements and iterations \cite{WaldspurgerPhase,chen2019gradient,tan2019online}. Nevertheless, such a similar result does not exist for the sparse phase retrieval problem. According to \cref{localconvergence}, to ensure the convergence and exact recovery of the underlying signal, a good initial guess is required for the proposed SAM algorithm.

%\cref{localconvergence}  requires a good initial guess, which is naturally true due to the nonconvexity of problem~\eqref{problem:alm}.
% is non-convex, a good initialization is of great importance for the convergence of the algorithm. Though for problems like standard phase retrieval with gaussian sensing matrix, arbitrary initial guess has been proved to be workable but at a cost of more measurements and iterations \cite{WaldspurgerPhase,chen2019gradient,tan2019online}. For the sparse phase retrieval problem, according to \cref{localconvergence}, to ensure the convergence and recovery the underlying signal efficiently, a good initial guess is required for the algorithm.

There are various strategies in the literature to design the initialization required by \cref{localconvergence}. For example, one can first estimate the support statistically, and then the signal on the support is computed by a spectral method. Here we follow the initialization technique introduced in \cite{jagatap2019sample}, where $\bm{x}_{0}$ is generated by the following two steps:
\begin{enumerate}\label{ini}
\item[(Init-1)]We first estimate the support of the initial guess. Notice that, for $j=1,2,\cdots,n$, we have
$$
\mathsf{E}\left[\frac{1}{m}\sum_{i=1}^{m}y_i^2a_{ij}^2\right]=\mathsf{E}\left[\frac{1}{m}\sum_{i=1}^{m}\left(\sum_{k=1}^n{x_k^{\natural}a_{ik}}\right)^2a_{ij}^2\right]=
\|\bm{x}^{\natural}\|_2^2+2(x_j^{\natural})^2.
$$
Thus, the support of top-$s$ entries in $\left\{\frac{1}{m}\sum_{i=1}^{m}y_i^2a_{ij}^2\right\}_{j=1}^{n}$ could be be a good approximation of the support of $\bm{x}^{\natural}$. So, we estimate the support of the initial guess by that of the top-$s$ entries in $\left\{\frac{1}{m}\sum_{i=1}^{m}y_i^2a_{ij}^2\right\}_{j=1}^{n}$, denoted by $\S_0$.
\item[(Init-2)] By noticing that $[\pm\bm{x}^{\natural}]_{\S_0}$ are principal eigenvectors of the expectation of $\frac{1}{m}\sum_{i=i}^my_i^2[\bm{a}_{i}]_{\S _0}[\bm{a}_{i}]_{\S _0}^T$, and $\|\bm{x}^{\natural}\|_2^2$ is the expectation of $\frac1m\|\bm{y}\|_2^2$. We let $[\bm{x}_0]_{\S_0}$ be a principal eigenvector of $\frac{1}{m}\sum_{i=i}^m y_i^2[\bm{a}_{i}]_{\S_0}[\bm{a}_{i}]_{\S_0}^T$ with length $\frac{1}{\sqrt{m}}\|\bm{y}\|_2$, and $[\bm{x}_0]_{{\S_0}^c}=\bm{0}$.
\end{enumerate}
The algorithm is summarized in \cref{alg:initialization}.

\begin{algorithm}[hbt!]
   \caption{(A Spectral Initialization \cite{jagatap2019sample})}\label{alg:initialization}
   \begin{algorithmic}[1]
   \STATE Input: The sensing matrix $\bm{A}\in \mathbb{R}^{m\times n}$, the observed data $\bm{y}$, the sparsity level $s$.
     \STATE Let $\S_0$ be the indices of top-$s$ entries in $\left\{\frac{1}{m}\sum_{i=1}^{m}y_i^2 a_{i,j}^2\right\}_{j=1}^{m}$.
     \STATE Obtain $\tilde{\bm{x}}_{0}$ by setting $[\tilde{\bm{x}}_{0}]_{{\S_0}^c}=\bm{0}$ and $[\tilde{\bm{x}}_{0}]_{\S_0}$ a unit principal eigenvector of the matrix $\frac{1}{m}\sum\limits_{i=i}^m y_i^2[\bm{a}_{i}]_{\S _0}[\bm{a}_{i}]_{\S _0}^T$.
     \STATE Compute $\bm{x}_{0}=\frac{1}{\sqrt{m}}\|\bm{y}\|_2\cdot\tilde{\bm{x}}_{0}$.
    \end{algorithmic}
\end{algorithm}

The following \cref{lemma:initialization} (which is also \cite[Theorem~IV.1]{jagatap2019sample}) shows that \cref{alg:initialization} produces a good initial guess under a suitable assumption on the sample size $m$.
\begin{lemma}[{\cite[Theorem IV.1]{jagatap2019sample}}]\label{lemma:initialization}
Let $\bm{x}^{\natural}\in\mathbb{R}^{n}$ be satisfying $\|\bm{x}^{\natural}\|_0\leq s$, and $\bm{A}\in\mathbb{R}^{m\times n}$ be a random matrix whose entries are i.i.d. drawn from $\N(0,1)$. Let $\bm{x}_0$ be generated by \cref{alg:initialization} with $\bm{y}=\lv \bm{A}\bm{x}^{\natural}\rv$. Then for any $\eta\in (0,1)$, there exists a positive constant $C_0$ depending only on $\eta$ such that: if provided $m\ge C_0 s^2\log n$, then with probability at least $1-8m^{-1}$ it holds that
\begin{align*}
\mathrm{dist}\left(\bm{x}_{0},\bm{x}^{\natural}\right)\le \eta\lV\bm{x}^{\natural}\rV_2.
\end{align*}
\end{lemma}

By combining the initialization (\cref{lemma:initialization}) and the local convergence (\cref{localconvergence}), we have the following recovery guarantee for our proposed SAM algorithm.
\begin{theorem}[Recovery guarantee of SAM]\label{convergence}
Let $\bm{x}^{\natural}\in\mathbb{R}^{n}$ be satisfying $\|\bm{x}^{\natural}\|_0\leq s$, and $\bm{A}\in\mathbb{R}^{m\times n}$ be a random matrix whose entries are i.i.d. drawn from $\N(0,1)$. Let $\{\bm{x}_{k}\}_{k\ge 0}$ be the iteration sequence produced by \cref{alg:SAM} with $\bm{y}=\lv \bm{A}\bm{x}^{\natural}\rv$, $\beta\in [\frac{1}{10},1]$, $L\geq 1$, and $\bm{x}_{0}$ generated by \cref{alg:initialization}. There exist universal constants $C_6, C_7, C_8>0$ such that: whenever $m\ge C_6 s^2\log(n/s)$,
\begin{enumerate}
\item[(a)] with probability at least $1-2Ke^{-C_8 m}-8m^{-1}$
\begin{align*}
\mathrm{dist}\left(\bm{x}_{k+1},\bm{x}^{\natural}\right)\le \alpha_0\cdot \mathrm{dist}\left(\bm{x}_{k},\bm{x}^{\natural}\right),\qquad\forall~k=0,1,2,\cdots,K-1.
\end{align*}
\item[(b)] if $L\geq 2s$, then with probability at least $1-2Ke^{-C_8 m}-9m^{-1}$,
\begin{align*}
\mathrm{dist}\left(\bm{x}_{k+1},\bm{x}^{\natural}\right)
\begin{cases}
\le \alpha_0\cdot \mathrm{dist}\left(\bm{x}_{k},\bm{x}^{\natural}\right), &\forall~k=0,1,2,\cdots,K-1,\cr
=0, &\forall~k\ge K-1.
\end{cases}
\end{align*}
for some $K \leq C_7\log m$.
\end{enumerate}
\end{theorem}
\begin{proof}
The theorem is a direct consequence of \cref{localconvergence} and \cref{lemma:initialization}. Since $\beta$ is bounded both above and below, the universal constants in \cref{localconvergence} can be combined with $\beta$ to obtain new universal constants, and the dependency of $C_0$ in \cref{lemma:initialization} on $\eta$ (and hence $\beta$) can be eliminated.
\end{proof}

%We remark that the lower bound $\beta$ in \cref{convergence} can be replaced by any positive fixed number. We choose $\beta\geq 0.6$, because we observed in our numerical experiments that $\beta=0.6$ achieves the best empirical performance.

By the same argument as \cref{convergence}, the recovery guarantee of \cref{alg:alm} can also be established, which ensures that \cref{alg:alm} with an initialization by \cref{alg:initialization} gives an exact recovery of $\bm{x}^{\natural}$ in at most $O(\log m)$ steps with high probability if provided $m\ge O(s^2\log n)$.

The required sample size $m= O(s^2\log n)$ in \cref{convergence} is dominated by the initialization stage, as \cref{lemma:initialization} requires $m= O(s^2\log n)$ for initialization and \cref{localconvergence} requires only $m= O(s\log n)$ for local convergence. It is possible to improve the sampling complexity by improving that of the initialization stage. There are several ways as in the following.
\begin{itemize}
\item
We may make further assumption on the distribution nonzero entries of $\bm{x}^{\natural}$. For instance, if the nonzero entries of the underlying signal $\bm{x}^{\natural}$ follow a power-law decay, then the sampling complexity of \cref{alg:initialization} can be reduced to $m=O(s\log n)$. See {\cite[Theorem IV.4]{jagatap2019sample}} for detailed discussions. By this way, the sampling complexity of our proposed SAM algorithm is improved to $O(s\log n)$.
\item
We may employ other initialization techniques, e.g., the one-step Hadamard Wirtinger flow introduced in \cite{wu2021hadamard} to produce an initial guess. Provided $x_{\min}^{\natural}\ge O(s^{-\frac12})$ and $m\ge O(s (x_{\max}^{\natural})^{-2}\log n)$, the Hardmard Wirtinger flow is able to produce an initial guess $\bm{x}_0$ satisfying
$\mathrm{dist}\left(\bm{x}_{0},\bm{x}^{\natural}\right)\le \eta\lV\bm{x}^{\natural}\rV_2$ for any positive $\eta$. See {\cite[Lemma 3]{wu2021hadamard}} and {\cite[Proposition 1]{wang2017solving}}. This initialization may lead a better theoretical sample complexity, but practically it requires multiple restarts, thus not as efficient as \cref{alg:initialization} in terms of computational cost.
\end{itemize}

\section{Numerical Results and Discussions}\label{section:experiments}

%\subsection{The Stopping Criteria}
%As the algorithm SAM converges at least linearly in the near neighbour of the underlying signal, so it is justifiable to set one stopping criteria of SAM in form
%$$\frac{\lV\bm{x}_{k+1}-\bm{x}_{k}\rV_2}{\lV\bm{x}_{k}\rV_2}\le \kappa$$
%for some small $\kappa$. The stopping criteria for SAM is set to be either $k\ge 200$ or $\lV\bm{x}_{k+1}-\bm{x}_{k}\rV /\lV\bm{x}_{k}\rV\le 10^{-3}$ in the experiments.
%\subsection{Numerical Simulation}
In this section, we present some numerical simulations of the proposed SAM algorithm and demonstrate its advantages over state-of-the-art algorithms for sparse phase retrieval.

\paragraph{Experimental Settings}
All numerical simulations are run on a laptop with an eight-core processor i$7-10870$H and $16$ GB memory using MATLAB $2021$a. The sampling vectors $\{\bm{a}_i\}_{i=1}^m$ are i.i.d. random Gaussian vectors with mean $\bm{0}$ and covariance matrix $\bm{I}$. The support of $\bm{x}^{\natural}$ are uniformly drawn from all $s$-subsets of $[n]$ at random, and its nonzero entries are randomly drawn from i.i.d. standard normal distribution. To achieve the best empirical performance, we set the maximum number of inner HTP iterations $L=3$ in \cref{alg:SAM} in all the experiments.

In experiments without noise, the parameters of stopping criteria for SAM is set as $K=200$ and $\epsilon=10^{-3}$ (i.e. $\lV\bm{x}_{k+1}-\bm{x}_{k}\rV_2 /\lV\bm{x}_{k}\rV_2\le 10^{-3}$).  An exact recovery is regarded if the output $\hat{\bm{x}}$ of the algorithm satisfies $\mathrm{dist}(\hat{\bm{x}},\bm{x}^{\natural})\le 10^{-3}\|\bm{x}^{\natural}\|_2$.

In experiments with noise, we still denote $\bm{y}$ the clean data ( i.e., $\bm{y}=|\bm{A}\bm{x}^{\natural}|$), and the noisy observed data $\bm{y}^{(\varepsilon)}$ is obtained by adding a Gaussian additive noise to $\bm{y}$ as in the following
$$
y_i^{(\varepsilon)}=y_i+\sigma\varepsilon_i,\quad i=1,\ldots,m,
$$
where $\varepsilon_i$ for $i=1,\ldots,m$ are randomly drawn from the standard normal distribution and $\sigma>0$ is the standard deviation. Therefore, the noise level is determined by $\sigma$. We set the parameters of stopping criteria for SAM as  $K=200$ and $\epsilon=10^{-3}+\sigma$ in the noisy case.

\paragraph{Experiment 1: Effect of random batch selection}
Stochasticity is a key feature of the proposed SAM algorithm to improve the empirical sampling complexity. In this experiment, we run the proposed SAM algorithm (\cref{alg:SAM}) with different Bernoulli probability $\beta$, i.e., $\beta=0.4,0.6,0.8,1$, respectively. We choose the signal dimension $n=1000$, the sparsity level $s=15, 30$ with various sample sizes $m$. For each set of parameters, we run $100$ independent trials on randomly chosen $\bm{A}$. The successful recovery rates are plotted in \cref{phasetrans_beta}.

From \cref{phasetrans_beta}, we see that \cref{alg:SAM} with $\beta=0.8, 0.6, 0.4$ has higher successful recovery rates than with $\beta=1$. This indicatess that the random batch technique in the proposed SAM algorithm can decrease the sample size $m$ empirically.

%Comparing $\beta=0.8, 0.6, 0.4$, the result of $\beta=0.6$ achieves the highest successful recovery rate. Therefore, a fine tuning of the level of stochasticity helps also reduce the number of measurements required in sparse phase retrieval.
%We plot the successful recovery rate by running $100$ independent trials with various sample size $m$. The successful recovery rate are obtained by $200$ independent trial runs. The \cref{phasetrans_beta} shows that
%for $\beta=0.8, 0.6, 0.4$, they have higher successful recovery rate than the setting $\beta=1$ when sample size $m$ is equal. Which means SAM requires less measurements than the standard alternating minimization method (\cref{alg:alm}).
\begin{figure}[!htb]
\centering
\subfigure[n=1000, s=15]{\includegraphics[clip=true,width=0.4\textwidth]{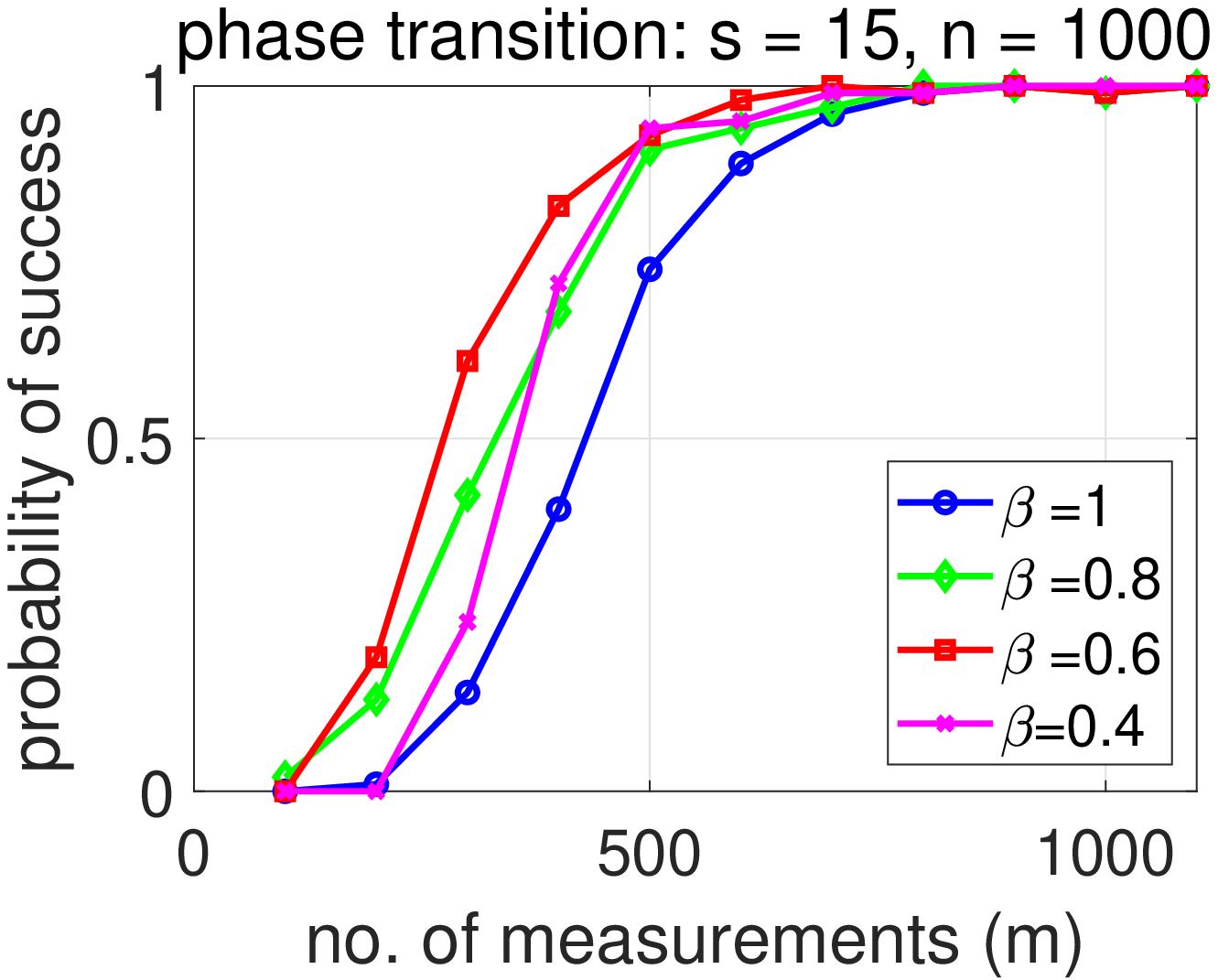}}
\subfigure[n=1000, s=30]{\includegraphics[clip=true,width=0.4\textwidth]{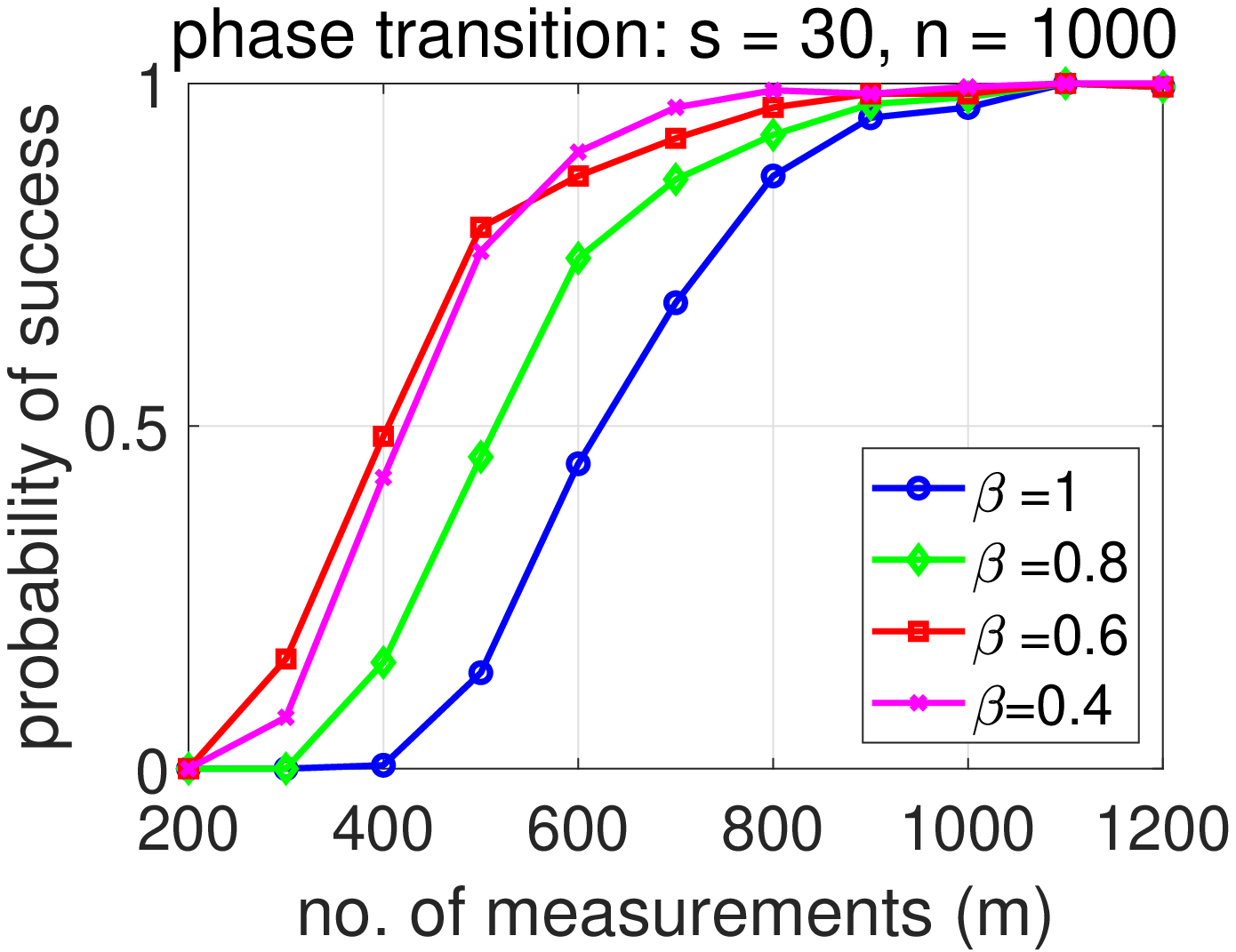}}
\caption{\label{phasetrans_beta}The successful recovery rates of \cref{alg:SAM} with different $\beta$. The signal dimension is $n=1000$, and the sparsity levels are $s=15$ (left) and $s=30$ (right). The successful recovery rates are obtained by $100$ independent trial runs.}
\end{figure}

\paragraph{Experiment 2: Comparison with state-of-the-art algorithms} We compare the proposed SAM algorithm with state-of-the-art algorithms, including CoPRAM \cite{jagatap2019sample}, Thresholded Wirtinger Flow (ThWF) \cite{cai2016optimal} and SPARse Truncated Amplitude flow (SPARTA) \cite{wang2016sparse}, in terms of sampling efficiency and running time. For the SAM algorithm, we fix $\beta=0.6$. For other algorithms, the parameters are set to the recommended values in the corresponding papers.
%For SPARTA, the parameters are set to be $\gamma=0.7,1=1,\lv\I\rv=\lfloor m/6\rfloor$ as suggested.

We first compare the number of measurements required by different algorithms. The signal dimension is fixed to be $n=1000$. For each set of parameters $(n,m,s)$ and each algorithm, we perform $100$ tests on randomly chosen $\bm{A}$. We plot in \cref{phasetrans_s} the successful recovery rates of different algorithms with the sparsity level $s=15,30$ and various sample sizes $m$. Moreover, \cref{PhaseTrans} depicts the phase transitions of different algorithms with various sparsity levels ranging from $s=5$ to $s=50$ and various sample size ranging from $m=100$ to $m=1200$. In this figure, a successful recovery rate is described by the gray level of the corresponding block: a white block represents a $100\%$ successful recovery rate, a black block $0\%$, and a grey block between $0\%$ and $100\%$. \cref{phasetrans_s} and \cref{PhaseTrans} show that SAM (\cref{alg:SAM}) requires less measurements than other algorithms for the same successful recovery rate. %Therefore, our SAM algorithm outperforms state-of-the-art sparse phase retrieval algorithms in terms of the sampling efficiency.

Next, we demonstrate the computational efficiency of the SAM algorithm compared with existing sparse phase retrieval algorithms. We fix the signal dimension $n=3000$ and the sample size $m=2000$. The relative error is defined to be
 $$
r\left(\hat{\bm{x}},\bm{x}^{\natural}\right)= \frac{\mathrm{dist}(\hat{\bm{x}},\bm{x}^{\natural})}{\Vert \bm{x}^{\natural}\Vert_2}.
$$
 \cref{tab:cputime} lists running times required and relative error achieved by different algorithms. Here each reported running time and mean relative error is the average over $100$ trial runs with failed ones filtered out. We see that the SAM algorithm is better than state-of-art algorithms in terms of running time.
\begin{figure}[!htb]
\centering
\subfigure[n=1000, s=15]{\includegraphics[clip=true,width=0.4\textwidth]{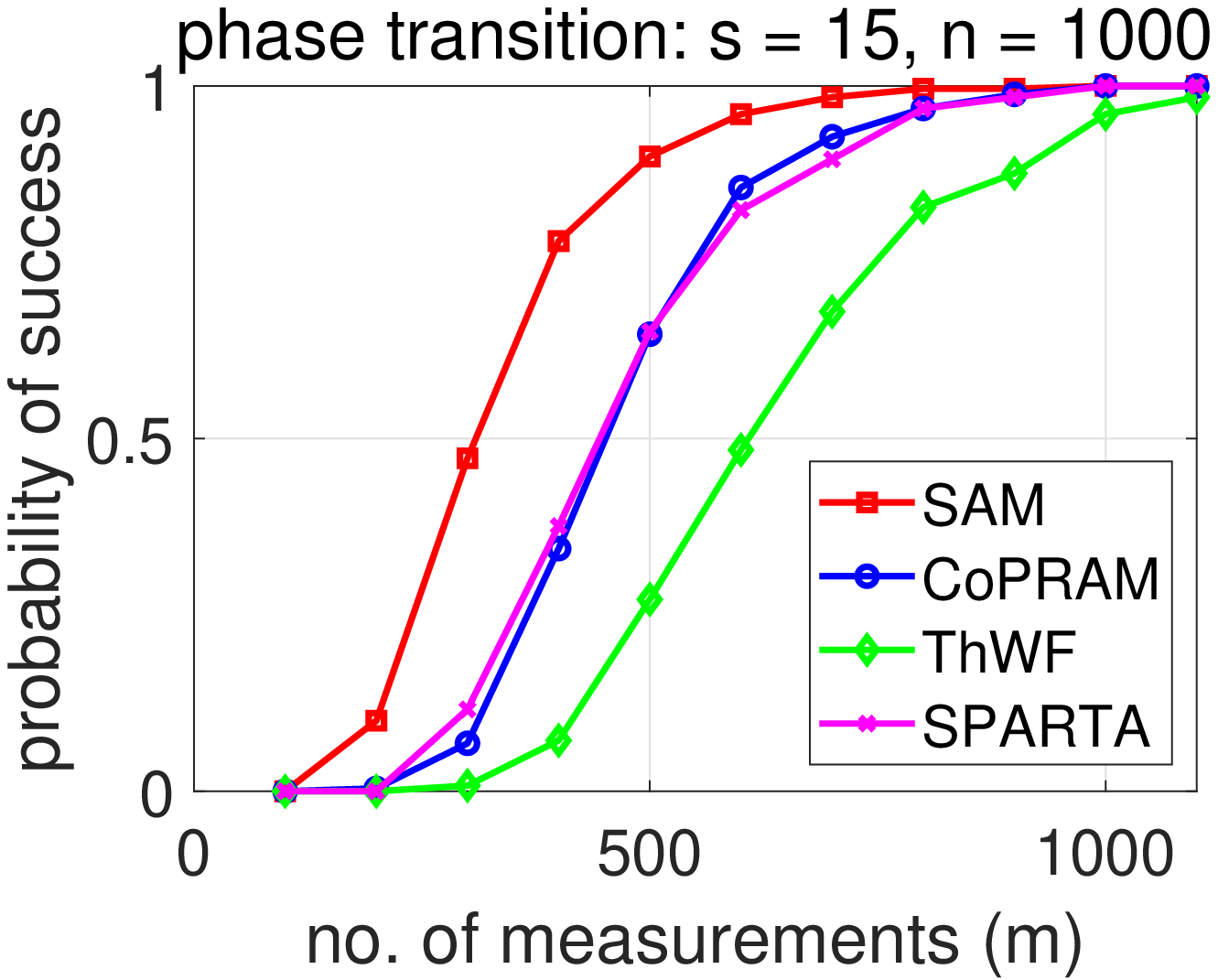}}
\subfigure[n=1000, s=30]{\includegraphics[clip=true,width=0.4\textwidth]{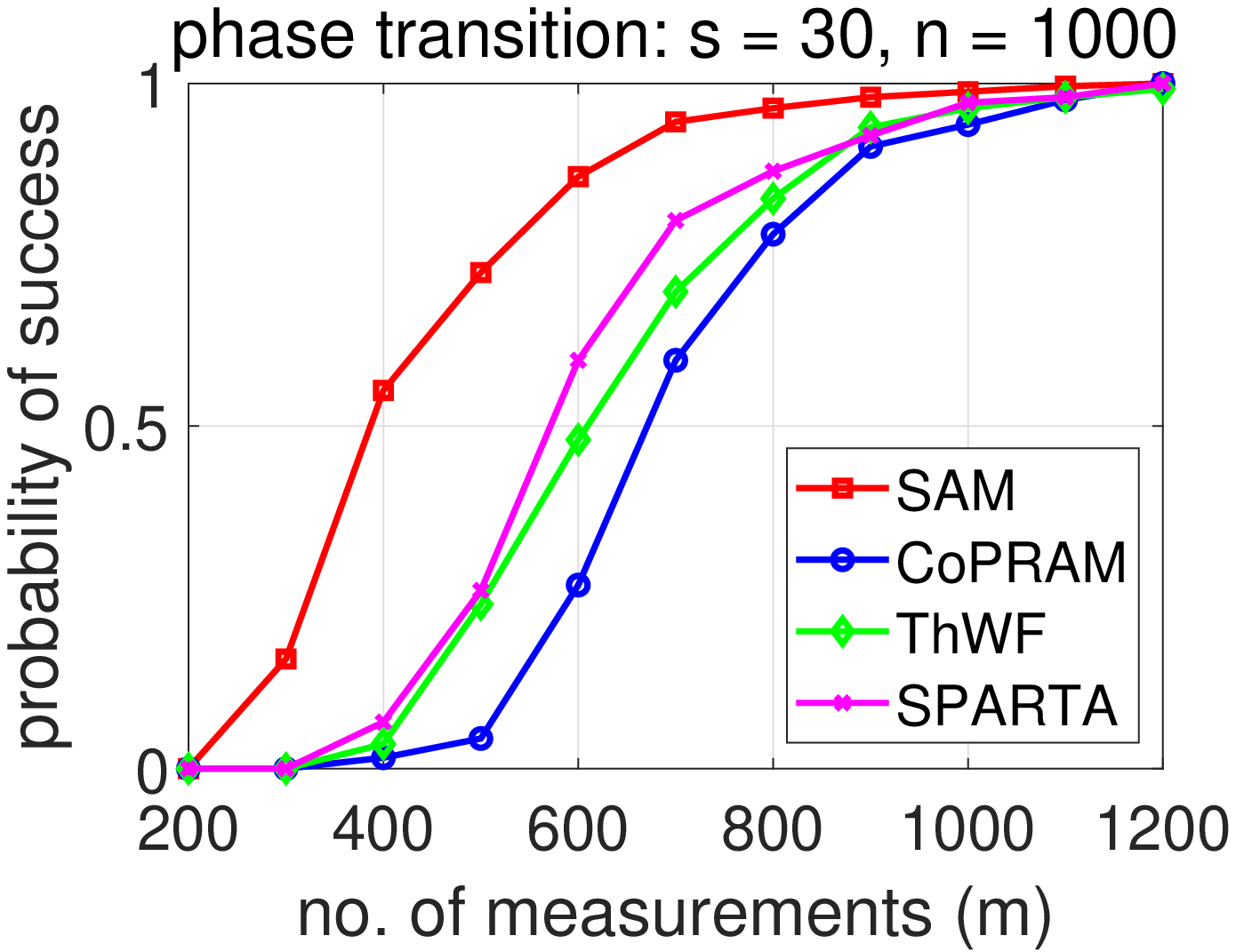}}
\caption{\label{phasetrans_s}The successful recovery rates of different algorithms. The signal dimension is $n=1000$, and the sparsity levels are $s=15$ (left) and $s=30$ (right).  The successful recovery rates are obtained by $100$ independent trial runs.}
\end{figure}
\begin{figure}[!htb]
\centering
\subfigure[ThWF]{\includegraphics[trim =1.5cm 0cm 1.5cm 0cm,clip=true,width=0.24\textwidth]{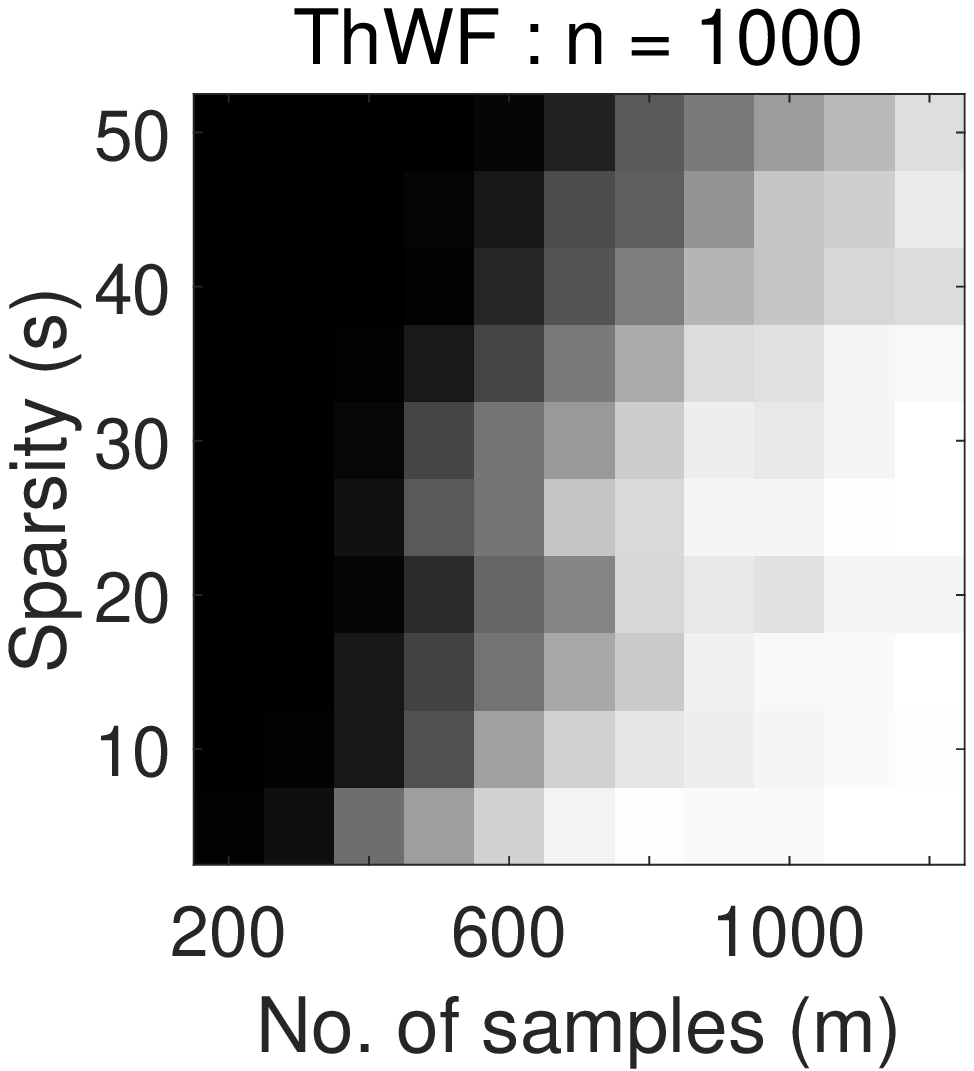}}
\subfigure[CoPRAM]{\includegraphics[trim =1.5cm 0cm 1.5cm 0cm,clip=true,width=0.24\textwidth]{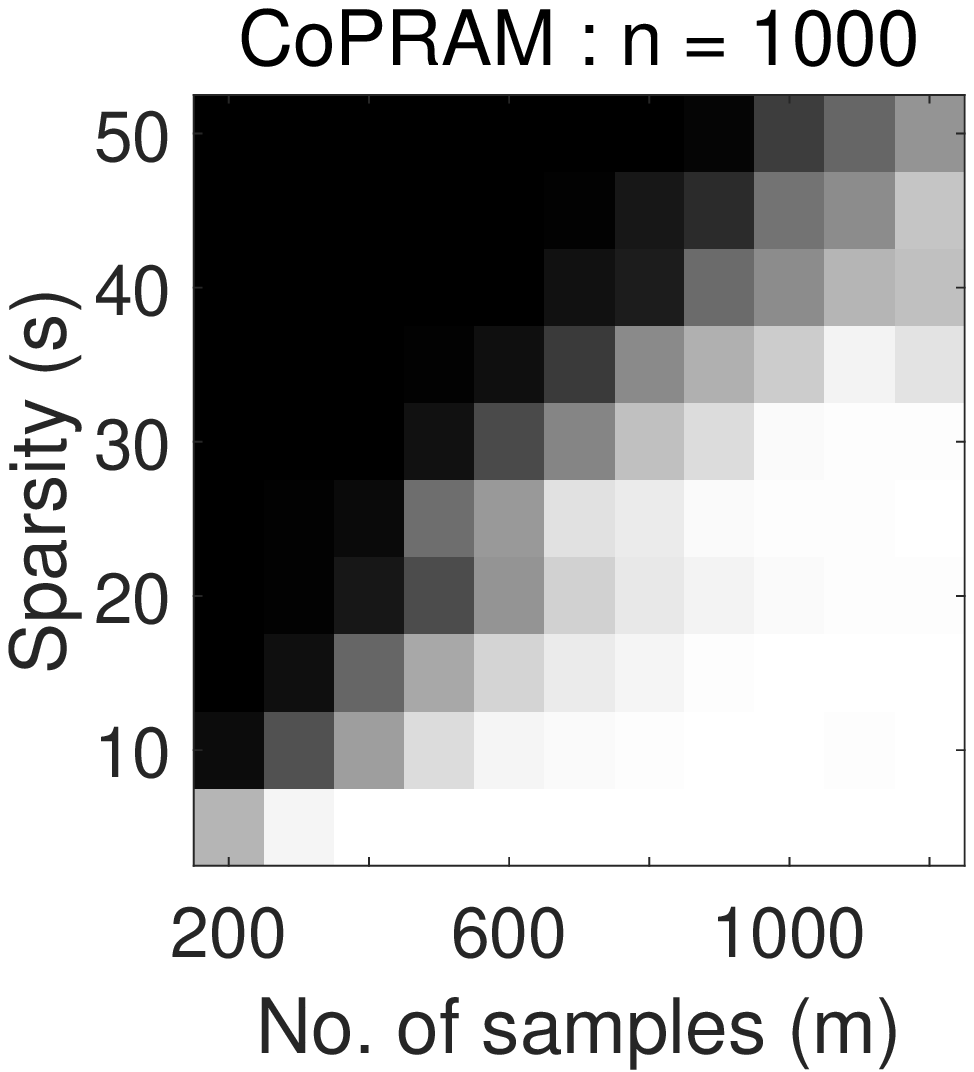}}
\subfigure[SPARTA]{\includegraphics[trim =1.5cm 0cm 1.5cm 0cm,clip=true,width=0.24\textwidth]{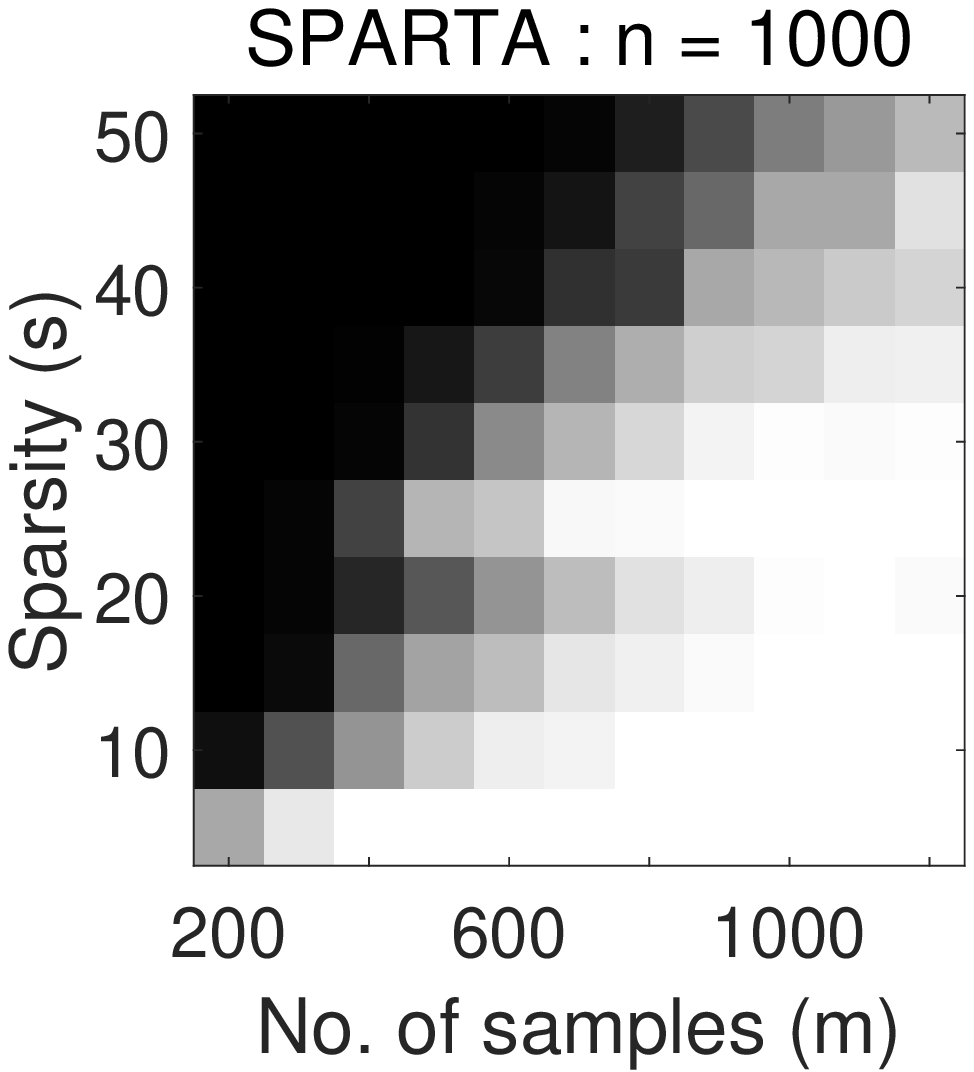}}
\subfigure[SAM]{\includegraphics[trim =1.5cm 0cm 1.5cm 0cm,clip=true,width=0.24\textwidth]{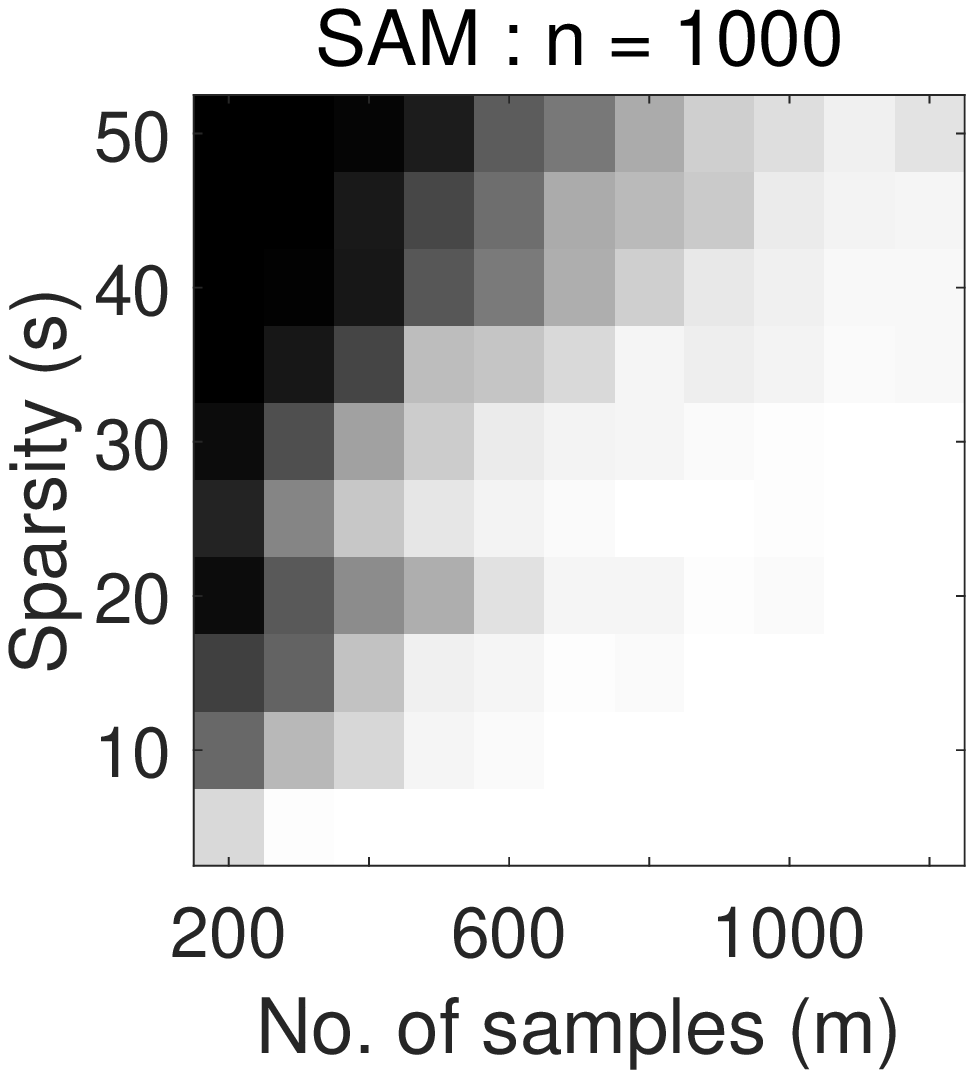}}
\caption{\label{PhaseTrans}Phase transition for different algorithms: signal dimension $n=1000$, the sparsity $s$ range from $5$ to $50$ with grid size $5$, and the sample size $m$ range from $200$ to $1200$ with grid size $100$. $\beta=0.6$ for SAM. The black block means successful recovery rate is $0\%$, the white block means successful recovery rate is $1$, and the grey block means successful recovery rate is between $0\%$ and $100\%$.}
\end{figure}
%\paragraph{Comparison with state-of-the-art algorithms: running time.} In this example, we compare SAM with  CoPRAM,ThWF and SPARTA, in terms of running time. For SPARTA, the parameters are set to be $\gamma=0.7,1=1,\lv\I\rv=\lfloor m/6\rfloor$ as suggested. For SAM, we set $\beta=0.6$. We fix the dimension $n=3000$ and the sample size $m=2000$, and \cref{tab:cputime} shows the mean running time required for a successful recovery. The mean running time are obtained by averaging $100$ independent trial runs with the failed recovery been filtered out. It demonstrates that SAM is comparable to state-of-art algorithms in terms of cpu time.

\begin{table}[htb!]
\footnotesize
\centering
\caption{SAM with state-of-the-art methods on running time in seconds (Time(s)). Mean relative error($r\left(\hat{\bm{x}},\bm{x}^{\natural}\right)$). Sparsity $s$, noise level $\sigma$. $m=2000, n=3000$.}\label{tab:cputime}
\begin{tabular}{cc|cc|cc|cc}
\hline\hline
&             &\multicolumn{2}{c|}{$s=20$, $\sigma=0$} &\multicolumn{2}{c|}{$s=30$, $\sigma=0$}      &\multicolumn{2}{c}{$s=40$, $\sigma=0$} \\
\hline
&Mehtod              &Time(s)   &$r\left(\hat{\bm{x}},\bm{x}^{\natural}\right)$ &Time(s)   &$r\left(\hat{\bm{x}},\bm{x}^{\natural}\right)$   &Time(s)  &$r\left(\hat{\bm{x}},\bm{x}^{\natural}\right)$  \\
 \hline
&ThWF  &$2.83\times 10^{-1}$  &$3.22\times 10^{-4}$  &$3.26\times 10^{-1}$  &$3.74\times 10^{-4}$                    &$3.92\times 10^{-1}$   &$5.07\times 10^{-4}$        \\
&SPARTA  &$2.58\times 10^{-1}$  &$5.56\times 10^{-4}$   &$3.29\times 10^{-1}$  &$6.08\times 10^{-4}$                    &$4.37\times 10^{-1}$  &$6.43\times 10^{-4}$  \\
&CoPRAM   &$1.21\times 10^{-1}$  &$1.04\times 10^{-4}$  &$1.63\times 10^{-1}$  &$1.82\times 10^{-4}$           &$2.31\times 10^{-1}$   & $2.33\times 10^{-4}$        \\
&SAM  &$1.08\times 10^{-1}$      &$8.65\times 10^{-8}$   &$1.22\times 10^{-1}$  &$3.41\times 10^{-7}$                      &$1.62\times 10^{-1}$   &$8.94\times 10^{-8}$        \\
\hline
&             &\multicolumn{2}{c|}{$s=20$, $\sigma=0.1$}   &\multicolumn{2}{c|}{$s=30$, $\sigma=0.1$} &\multicolumn{2}{c}{$s=40$, $\sigma=0.1$} \\
\hline
&Mehtod              &Time(s)   &$r\left(\hat{\bm{x}},\bm{x}^{\natural}\right)$   &Time(s)  &$r\left(\hat{\bm{x}},\bm{x}^{\natural}\right)$  &Time(s)  &$r\left(\hat{\bm{x}},\bm{x}^{\natural}\right)$  \\
 \hline
&ThWF  &$1.43\times 10^{-1}$  &$3.83\times 10^{-2}$    &$1.57\times 10^{-1}$  &$4.61\times 10^{-2}$                  &$1.86\times 10^{-1}$   &$5.34\times 10^{-2}$        \\
&SPARTA  &$2.83\times 10^{-1}$  &$1.27\times 10^{-2}$   &$3.41\times 10^{-1}$  &$1.59\times 10^{-2}$                    &$4.63\times 10^{-1}$  &$2.09\times 10^{-2}$  \\
&CoPRAM   &$1.41\times 10^{-1}$  &$1.31\times 10^{-2}$   &$1.89\times 10^{-1}$  &$1.70\times 10^{-2}$          &$2.86\times 10^{-1}$   & $2.27\times 10^{-2}$        \\
&SAM  &$7.42\times 10^{-2}$      &$1.90\times 10^{-2}$   &$8.94\times 10^{-2}$  &$2.07\times 10^{-2}$                      &$1.39\times 10^{-1}$   &$2.77\times 10^{-2}$        \\
\hline\hline
\end{tabular}
\end{table}
\paragraph{Experiment 3: Robustness to additive noise}
Although the theoretical results are for noiseless measurements only, the proposed SAM algorithm also works well for noisy data, which is demonstrated by the following experiment. We test the performance of \cref{alg:SAM} in the presence of an additive noise.  We then recover the sparse signal from $\bm{y}^{(\varepsilon)}$ by SAM. We set $n=5000$, $m=1500$, $s=20$, and we test SAM with $\beta=0.6$ under different noise level $\sigma$. In \cref{noiseRobust}, we plot the mean relative error by our SAM algorithm against the signal-to-noise ratios of $\bm{y}^{(\varepsilon)}$. The mean relative error are obtained by averaging $100$ independent trial runs with the failed recovery filtered out. We see from \cref{noiseRobust} that SAM is robust to the additive noise in the measurements.
\begin{figure}[!htb]
\centering
{\includegraphics[clip=true,width=0.5\textwidth]{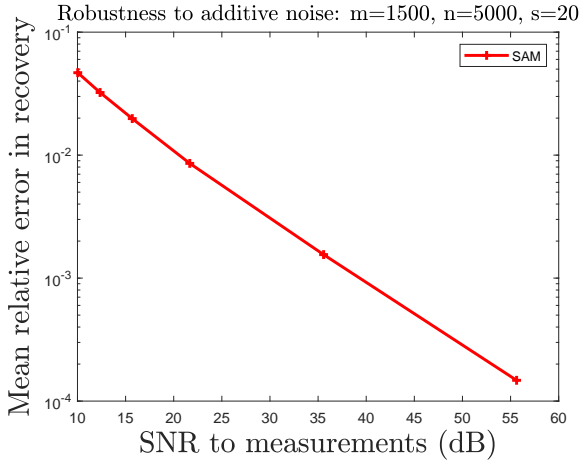}}
\caption{\label{noiseRobust} Mean relative error in the recovery ($\log_{10}$) vesus signal-to-noise ratios (SNR) of the measurements data. We set $n=5000, m=1500,s=20$.}
\end{figure}
%\JT{\begin{remark}
%Numerical results suggest that SAM is able to recovery the signal from very few number of measurements---though the initial guess is not sufficiently close to the true signal, the algorithm may still recovery the signal successfully, but more iterations may be required, see \cref{fewsample}. Moreover, random initialization is possible to recovery the underlying signal successfully using SAM, experiments for this can refer to \cite{}.
%\end{remark}}
% \begin{figure}[!htb]
%\centering
%{\includegraphics[clip=true,width=0.5\textwidth]{single_test.eps}}
%\caption{\label{fewsample} SAM: $n=3000, m=600,s=30,\beta=0.6$. The result is obtained by a single run.}
%\end{figure}

\section{Proofs}\label{section:proofs}
In this section, we present proofs of our main results \cref{localconvergence} and \cref{converge:alm}. We first give some key lemmas and their proofs in \cref{subsec:lemmas}. Then, we prove \cref{localconvergence} and \cref{converge:alm} in \cref{subsec:localconvergence} and \cref{subsec:proofalm} respectively. To make the paper self-contained, we also provide in \cref{subsec:supportlemmas} some supporting lemmas (Lemmas~\ref{ARIP}--\ref{lemma:concentration}) from the literature .

\subsection{Key Lemmas}\label{subsec:lemmas}

In this subsection, we give some lemmas that play key roles in proofs of our main results.

In the proposed SAM algorithm, at the $k$-th iteration, we randomly pick a subset $\I_k$ from the set $[m]$ without repeat elements by using the Bernoulli model. When $|\I_k|$ is as large as $O\left(s\log(n/s)\right)$, it can be shown that the coefficient matrix $\bm{A}_{k}:=\bm{A}(\I_k,:)$ satisfies the restricted isometric property (RIP). It is well known that the RIP is a key condition in many algorithms and theory of compressed sensing. In the following \cref{sRIP}, we show that the coefficient matrices $\bm{A}_{1},~\bm{A}_{2},~\ldots,\bm{A}_{K}$ in the first $K$ iterations of \cref{alg:SAM} satisfies the RIP simultaneously. The lemma is an extension of the RIP of standard gaussian matrix, and it is crucial for the convergence analysis of SAM.

\begin{lemma}[Simultaneous RIP]\label{sRIP}
Let the sensing vectors $\{\bm{a}_i\}_{i=1}^m$ be Gaussian random vectors that are $i.i.d.$ sampled from the normal distribution $\N\left(\bm{0},\bm{I}\right)$. Let $r\le n$ be a given positive integer. Let $\I_1,\I_2,\cdots,\I_K$ be the $K$ random subsets of $\{1,2,\ldots,m\}$ generated by Step 5 of \cref{alg:SAM}. Define $\bm{A}_k=\bm{A}(\I_k,:)$ for any $k$.  Then, for any $\delta\in(0,1)$, there exists constants $c_1,c_2>0$ such that: if provided $m\ge c_1 \beta^{-2}r\log\left(n/r\right)$, with probability at least $1-2Ke^{-c_2 \beta^2 m}$ it holds that
\begin{align}\label{srip:union}
\left(1-\delta\right)\lV\bm{x}\rV_2^2\le \frac{1}{\beta m}\lV\bm{A}_{k}\bm{x}\rV^2_2\le \left(1+\delta\right)\lV\bm{x}\rV_2^2,
\qquad\forall~\bm{x}:\lV\bm{x}\rV_0\le r\mbox{ and }k=1,2,\ldots,K.
\end{align}
\end{lemma}
\begin{proof}
We first prove the case of $K=1$. Since $\I_1$ is Bernoulli sampled from $[m]$ with probability $\beta$, its characteristic vector $\bm{\xi}=[\xi_1,\xi_2,\cdots,\xi_m]^T$ is a random vector whose entries are independent and satisfy
%\begin{align}\label{def:delta}
%\begin{split}
%\xi_i=\left\{
%\begin{aligned}
%&1,\quad &\text{with probability}\ \beta, \\
%&0,      &\text{with probability}\ 1-\beta, \\
%\end{aligned}
%\right.
%\qquad i=1,\ldots,m.
%\end{split}
%\end{align}
\begin{equation}  \label{def:delta}
 \xi_i=
 \begin{cases}
 1,\quad &\textnormal{with probability}\ \beta; \cr
 0,      &\textnormal{otherwise},
 \end{cases}
\quad i=1,2,\cdots,m.
 \end{equation}
It suffices to show that: for any $\delta\in(0,1)$, if $m\ge c_1 \beta^{-2}r\log\left(n/r\right)$, then
\begin{equation}\label{eqsrip}
\mathsf{P}\left( \lv \frac{1}{\beta m}\sum_{i=1}^m\xi_i \lv\bm{a}_i^T\bm{x}\rv^2-\lV\bm{x}\rV_2^2\rv>\delta\lV\bm{x}\rV_2^2,~~\forall~\bm{x}:\|\bm{x}\|_0\leq r\right)\le 2e^{-c_2\beta^2m},
\end{equation}
where $\xi_1,\xi_2,\cdots,\xi_m$ are independent random Bernoulli variables. To this end, we use the same argument as in the proof of RIP for random Gaussian matrices.

Let $\bm{x}\in\mathbb{R}^n$ be a fixed vector. Obviously, we have \begin{equation*}
\frac{1}{\beta m}\lV\bm{A}_{1}\bm{x}\rV^2_2=\frac{1}{\beta m}\sum_{i=1}^m\xi_i \lv\bm{a}_i^T\bm{x}\rv^2.
\end{equation*}
Since the random vector $\bm{\xi}$ and the random matrix $\bm{A}$ are independent, taking full expectation leads to
\begin{align*}
\mathsf{E}\left[\frac{1}{\beta m}\lV\bm{A}_{1}\bm{x}\rV^2\right]
=\mathsf{E}\left[\frac{1}{\beta m}\sum_{i=1}^m\xi_i \lv\bm{a}_i^T\bm{x}\rv^2\right]
&=\mathsf{E}_{\bm{A}}\mathsf{E}_{\bm{\xi}}\left[\frac{1}{\beta m}\sum_{i=1}^m\xi_i \lv\bm{a}_i^T\bm{x}\rv^2\right]\\
&=\mathsf{E}_{\bm{A}}\left[\frac{1}{m}\sum_{i=1}^m \lv\bm{a}_i^T\bm{x}\rv^2\right]
=\lV\bm{x}\rV_2^2.
\end{align*}
Denote
\begin{equation*}
  v_i=\frac{\xi_i}{\beta}\lv\bm{a}_i^T\bm{x}\rv^2-\lV\bm{x}\rV_2^2,\quad i\in[m].
\end{equation*}
Then $v_1, v_2,\cdots, v_m$ are independent and $\mathsf{E}[v_i]=0$ for all $i\in [m]$. Moreover, for all $i\in [m]$, since $\xi_i$ is bounded and $\bm{a}_i^T\bm{x}$ is Gaussian, one can show that $v_i$ is subexponential. Indeed, since $\bm{a}_{i}\sim \N(0,\bm{I})$, $\bm{a}_i^T\bm{x}$ is a Gaussian random variable with mean zero and variance $\|\bm{x}\|_2^2$, which implies that
$$\mathsf{P}\left( \lv\bm{a}_i^T\bm{x}\rv\ge \sqrt{\beta(1+\varepsilon)} \lV\bm{x}\rV_2\right)\le 2e^{-\frac{\beta(1+\varepsilon)}{2}},\quad\forall~ \varepsilon\geq 0.$$
We then have
\begin{align*}
&~\mathsf{P}\left(\lv v_i\rv \ge \varepsilon \lV\bm{x}\rV_2^2 \right)\\
=&~\mathsf{P}\left( \frac{\xi_i}{\beta}\lv\bm{a}_i^T\bm{x}\rv^2-\lV\bm{x}\rV_2^2\ge \varepsilon \lV\bm{x}\rV_2^2\right)+\mathsf{P}\left(\frac{\xi_i}{\beta}\lv\bm{a}_i^T\bm{x}\rv^2-\lV\bm{x}\rV_2^2 \le -\varepsilon \lV\bm{x}\rV_2^2\right)\\
\le&~ \beta\cdot\mathsf{P}\left( \lv\bm{a}_i^T\bm{x}\rv\ge \sqrt{\beta(\varepsilon+1)} \lV\bm{x}\rV_2\right)+\mathsf{P}\left(-\lV\bm{x}\rV_2^2 \le -\varepsilon \lV\bm{x}\rV_2^2\right)\\
\le &~2\beta e^{-\frac{\beta(1+\varepsilon)}{2}}+e^{\frac{\beta\left(1-\varepsilon\right)}{2}}
=\left(2\beta e^{-\frac{\beta}{2}}+e^{\frac{\beta}{2}}\right)e^{-\frac{\beta\varepsilon}{2}},
\end{align*}
where
%the second inequality follows from the inequality $-\lV\bm{z}\rV_2^2\le\frac{\xi_i}{\beta}\lv\l\bm{a}_i, \bm{z}\r\rv^2-\lV\bm{z}\rV_2^2\le \frac{1}{\beta}\lv\l\bm{a}_i, \bm{z}\r\rv^2-\lV\bm{z}\rV_2^2$,
the second inequality follows from
$$
\mathsf{P}\left(-\lV\bm{x}\rV_2^2 \le -\varepsilon \lV\bm{x}\rV_2^2\right)=
\begin{cases}
1, & \mbox{if }\varepsilon\le 1,\cr
0, & \mbox{if }\varepsilon> 1,\cr
\end{cases}
\leq e^{\frac{\beta\left(1-\varepsilon\right)}{2}}.
$$
Therefore, for any $u\ge 0$,
$$
\mathsf{P}\left(\lv v_i\rv \ge u\right)\le c_3e^{-c_4 u},\quad \text{with}~c_3=2\beta e^{-\frac{\beta}{2}}+e^{\frac{\beta}{2}},~c_4=\frac{\beta}{2\lV\bm{x}\rV_2^2},
$$
which tells that $u_i$ is subexponential. By applying the Bernstein's inequality (see also \cref{Bernstein}), it yields that, for any $\varepsilon\in(0,1)$,
\begin{align*}
&~\mathsf{P}\left(\frac{1}{m}\lv\sum_{i=1}^m v_i\rv \ge \varepsilon\lV\bm{x}\rV_2^2\right)
=\mathsf{P}\left(\lv\sum_{i=1}^m v_i\rv \ge m\varepsilon\lV\bm{x}\rV_2^2\right)\\
\le&~ 2\cdot\exp\left( -\frac{(c_4m\varepsilon \lV\bm{x}\rV_2^2)^2/2}{2c_3 m+c_4 m\varepsilon \lV\bm{x}\rV_2^2}\right)\\
=&~2\cdot\exp\left( \frac{-m \beta^2\varepsilon^2}{4(8\beta e^{-\frac{\beta}{2}}+4e^{\frac{\beta}{2}}+\beta\varepsilon)}\right)\le 2\cdot\exp\left( \frac{-m \beta^2\varepsilon^2}{36+16\sqrt{e}}\right),
\end{align*}
where the last inequality comes from $\beta e^{-\frac{\beta}{2}}\leq 1$, $e^{\frac{\beta}{2}}\leq\sqrt{e}$, and $\beta\varepsilon\leq 1$. Letting $c_0=\frac{1}{36+16\sqrt{e}}\approx 0.016$, we obtain, for all $\bm{x}\in\mathbb{R}^n$ and $\varepsilon\in(0,1)$,
\begin{align}\label{concentration}
\mathsf{P}\left( \lv \frac{1}{\beta m}\lV\bm{A}_{1}\bm{x}\rV_2^2-\lV\bm{x}\rV_2^2\rv>\varepsilon\lV\bm{x}\rV_2^2\right)\le 2e^{-c_0\varepsilon^2 \beta^2m}.
\end{align}

With \eqref{concentration}, we may follow the standard covering argument (see, e.g., \cite[Theorem 5.2]{Baraniuk2008A}) to prove the lemma with $K=1$. To make the paper self-contained, we provide the argument briefly.
Firstly, let $\S$ be any fixed subset $\S\subseteq [n]$ with $|\S|= r$, and define the subspace $\mathbb{B}_{\S}=\{\bm{x}\in \mathbb{R}^n: \mathrm{support}(\bm{x})\subseteq \S\}$. Then, by \cref{lemma:concentration} and \eqref{concentration}, we know for any $\tilde{\delta}\in (0,\frac{1}{3})$,  the inequality
\begin{align*}
(1-\tilde{\delta})\lV\bm{x}\rV_2\le \frac{1}{\sqrt{\beta m}}\lV\bm{A}_{1}\bm{x}\rV_2
\le (1+\tilde{\delta})\lV\bm{x}\rV_2,\quad~\forall~\bm{x}:\bm{x}\in \mathbb{B}_{\S}
\end{align*}
fail to hold with probability at most $2(12/\tilde{\delta})^r e^{-c_0(\tilde{\delta}/2)^2\beta^2 m}$. Since there are $n\choose r$ possible such subspaces (in form of $\mathbb{B}_{\S}$), the fail probability of the inequality
\begin{align}\label{ineq:ripJ1s}
(1-\tilde{\delta})\lV\bm{x}\rV_2\le \frac{1}{\sqrt{\beta m}}\lV\bm{A}_{1}\bm{x}\rV_2\le (1+\tilde{\delta})\lV\bm{x}\rV_2,\quad~\forall~\bm{x}:\lV\bm{x}\rV_0\le r
\end{align}
is at most
\begin{align*}
2{n\choose r}(12/\tilde{\delta})^r e^{-c_0(\tilde{\delta}/2)^2\beta^2 m}
&\le 2(en/r)^r(12/\tilde{\delta})^r e^{-c_0(\tilde{\delta}/2)^2\beta^2m}\\
&=2e^{-c_0(\tilde{\delta}/2)^2\beta^2 m+r\big(\log(en/r)+\log(12/\tilde{\delta})\big)},
\end{align*}
where the inequality follows from ${n\choose r}\le (en/r)^r$. By letting $c_1:=\frac{8\big(2+\log(12/\tilde{\delta})\big)}{c_0\tilde{\delta}^2}$, we have
$$
c_0\tilde{\delta}^2/4-c_1^{-1}\big(1+\frac{1+\log(12/\tilde{\delta})}{\log (n/r)}\big)\geq c_0\tilde{\delta}^2/8:=c_2.
$$
Therefore, whenever $r\le c_1^{-1} \beta^{2}m /\log\left(n/r\right)$, it holds that
$$
-c_0(\tilde{\delta}/2)^2\beta^2 m+r\big(\log(en/r)+\log(12/\tilde{\delta})\big)\leq -c_2 \beta^2m.
$$
This implies that, if provided $m\ge c_1 \beta^{-2}r\log\left(n/r\right)$, then the fail probability of \eqref{ineq:ripJ1s} is at most $2e^{-c_2 \beta^2m}$. Now, we set $\delta=3\tilde{\delta}$. Since $(1-\tilde{\delta})^2\ge 1-3\tilde{\delta}$ and $(1+\tilde{\delta})^2\le 1+3\tilde{\delta}$ for any $\tilde{\delta}\in (0,\frac{1}{3})$, \eqref{ineq:ripJ1s} implies
\begin{align*}
\left(1-\delta\right)\lV\bm{x}\rV_2^2\le \frac{1}{\beta m}\lV\bm{A}_{1}\bm{x}\rV_2^2\le \left(1+\delta\right)\lV\bm{x}\rV_2^2,\quad\forall~\bm{x}: \lV\bm{x}\rV_0\le r,
\end{align*}
for any $\delta\in (0,1)$. This proves \eqref{eqsrip} (i.e., the lemma for $K=1$).

Finally, we prove the lemma for a general $K$ by simply considering the union bound.  More explicitly, for any fixed $k\in \{1,\ldots,K\}$, the result of the case $K=1$ (i.e., \eqref{eqsrip}) implies that the fail probability of
\begin{equation}\label{eq:event1}
\left(1-\delta\right)\lV\bm{x}\rV_2^2\le \frac{1}{\beta m}\lV\bm{A}_{k}\bm{x}\rV^2_2\le \left(1+\delta\right)\lV\bm{x}\rV_2^2,\quad\forall~\bm{x}:\lV\bm{x}\rV_0\le r
\end{equation}
is at most $2e^{-c_2 \beta^2m}$. Thus, the fail probability of the event \eqref{eq:event1} for all $k\in\{1,\ldots,K\}$ would not exceed $2Ke^{-c_2 \beta^2m}$.
\end{proof}

%We shall notice that the sparsity $r$ in the above lemma can be replaced by an arbitrary sparsity level $\gamma$ (e.g. $\gamma=2s,3s$).

The following probabilistic lemma is also crucial for the proof of our main theorem in bounding the term $\lV \bm{y}_{k+1}-\bm{A}_{k+1}\bm{x}^{\natural}\rV_2$.

\begin{lemma}[A corollary of {\cite[Lemma 25]{soltanolkotabi2019structured}}]\label{bound:Ax}
Assume the sampling vectors $\left\{\bm{a}_i\right\}_{i=1}^{m}$ are i.i.d. Gaussian random vectors distributed as $\N(\bm{0},\bm{I})$. Assume $\bm{x}^{\sharp}$ is an $s$-sparse vector. There exist universal positive constants $c_5, c_6$ such that: as long as the sample size $m$ satisfies
$$m\geq c_5 s\log\left(n/s\right),$$
then with probability at least $1-e^{-c_6 m}$, it holds that
\begin{align}\label{event:boundAx}
\frac{1}{m}\mathop{\sum}\limits_{i=1}^{m}\lv \bm{a}_i^T\bm{x}^{\natural}\rv^2\cdot \bm{1}_{\left\{\left(\bm{a}_i^T\bm{x}\right)\left(\bm{a}_i^T\bm{x}^{\natural}\right)\le 0\right\}}
&\le\frac{1}{\left(1-\lambda\right)^2}\left(10^{-3}+\lambda \sqrt{\frac{21}{20}}\right)^2\lV\bm{x}-\bm{x}^{\natural}\rV_2^2,\cr
\quad
&\forall~\bm{x}~:~\|\bm{x}\|_0\leq s\mbox{~and~}\mathrm{dist}\left(\bm{x},\bm{x}^{\natural}\right)\le \lambda \lV\bm{x}^{\natural}\rV_2.
\end{align}
\end{lemma}
\begin{proof}
In fact, the left hand side of the inequality~\eqref{event:boundAx} is same to the second line of \cite[Eq. (VIII.45)]{soltanolkotabi2019structured}, and the upper bound of the term has been given by \cite[Lemma 25]{soltanolkotabi2019structured} with $\varepsilon_0 = 10^{-3}$.
\end{proof}

With the two probabilistic lemmas above, we can show the following deterministic lemmas under the success of the events \eqref{srip:union} and \eqref{event:boundAx}.

\begin{lemma}\label{bound:Ax-y}
Let the sequences $\{ \bm{y}_k, \bm{A}_{k},  \bm{x}_{k}\}_{k\ge 1}$ be generated by \cref{alg:SAM}. Assume the event~\eqref{event:boundAx} holds true for some $\lambda\in [0,\frac{1}{8}]$. Then, if $\lV \bm{x}_{k}-\bm{x}^{\natural}\rV_2\le \lambda \lV \bm{x}^{\natural}\rV_2$, we have
\begin{align}\label{bound:sign}
\lV \bm{y}_{k+1}-\bm{A}_{k+1}\bm{x}^{\natural}\rV_2
\le C_{\lambda}\sqrt{m}\lV \bm{x}_{k}-\bm{x}^{\natural}\rV_2,
\end{align}
where $C_{\lambda}=\frac{2}{(1-\lambda)}\left(10^{-3}+\lambda \sqrt{\frac{21}{20}}\right)$.
% with $\varepsilon_0=10^{-3}\sqrt{\beta}$.
\end{lemma}
\begin{proof}
Recall that $\bm{y}_{k+1}=\mathrm{sgn}(\bm{A}_{k+1}\bm{x}_{k})\odot \bm{y}_{\I_{k+1}}$. We thus have
\begin{align*}
\frac{1}{m}\lV \bm{y}_{k+1}-\bm{A}_{k+1}\bm{x}^{\natural}\rV_2^2 &=\frac{1}{ m} \mathop{\sum}\limits_{i\in \I_{k+1}}\Big(\lv\bm{a}_i^T\bm{x}^{\natural}\rv\cdot\mathrm{sgn}\left(\bm{a}_i^T\bm{x}_{k}\right)-\left(\bm{a}_i^T\bm{x}^{\natural}\right)\Big)^2\\
&\le\frac{1}{ m}\mathop{\sum}\limits_{i=1}^{m}\Big(\lv\bm{a}_i^T\bm{x}^{\natural}\rv\cdot\mathrm{sgn}\left(\bm{a}_i^T\bm{x}_{k}\right)-\left(\bm{a}_i^T\bm{x}^{\natural}\right)\Big)^2\\
&=\frac{1}{ m}\mathop{\sum}\limits_{i=1}^{m}\Big(\mathrm{sgn}\left(\bm{a}_i^T\bm{x}_{k}\right)-\mathrm{sgn}\left(\bm{a}_i^T\bm{x}^{\natural}\right)\Big)^2
\lv\bm{a}_i^T\bm{x}^{\natural}\rv^2 \\
&\le  \frac{4}{ m}\mathop{\sum}\limits_{i=1}^m \lv \bm{a}_i^T\bm{x}^{\natural}\rv^2 \cdot \bm{1}_{\left\{\left(\bm{a}_i^T\bm{x}_{k}\right)\left(\bm{a}_i^T\bm{x}^{\natural}\right)\le 0\right\}}\\
&\le \frac{4}{(1-\lambda)^2}\left(10^{-3}+\lambda \sqrt{\frac{21}{20}}\right)^2\lV \bm{x}_{k}-\bm{x}^{\natural}\rV_2^2,
\end{align*}
where the last inequality follows from \eqref{event:boundAx}. We conclude the proof by letting $C_{\lambda}=\frac{2}{(1-\lambda)}\left(10^{-3}+\lambda \sqrt{\frac{21}{20}}\right)$.
%we
%\begin{align*}
%\lV \bm{y}_{k+1}-\bm{A}_{k+1}\bm{x}^{\natural}\rV_2
%\le C(\lambda)\sqrt{m}\lV \bm{x}_{k}-\bm{x}^{\natural}\rV_2.
%\end{align*}
\end{proof}

In the following \cref{convergence:exact}, we consider the case when subproblem~$\eqref{subproblem:sto}$ is solved by HTP, in view of results from compressed sensing problem with noisy data. By \cref{sRIP}, one fact we shall notice is that if $m$ is $O\left(s\log(n/s)\right)$, then $|\I_k|$ should also be $O\left(s\log(n/s)\right)$ to ensure the RIP condition. Therefore, $\beta$ can not approach $0$ and a lower bound of $\beta$ is essential in practice. Without loss of generality, we consider $\beta\in [\frac{1}{10},1]$.

\begin{lemma}\label{convergence:exact}
Let the sequences $\{ \bm{y}_k, \bm{A}_{k},  \bm{x}_{k}\}_{k\ge 1}$ be generated by \cref{alg:SAM} with $L\geq 1$. Let $K$ be a given positive integer. Assume the simultaneous RIP \eqref{srip:union} holds true for $K$ iterations with $r=2s$ and $\delta=0.1$, and the event \eqref{event:boundAx} holds true for some $\lambda\in [0,{\sqrt{\beta}}/{8}]$ with $\beta\in [\frac{1}{10},1]$.
%and $\{\bm{x}_{k,\ell}\}_{\ell\ge 0}$ is the sequence generated by HTP with $L\ge 3$ solving the subproblem
%$$\mathop{\mathrm{minimize}}\limits_{ \lV \bm{x}\rV_0\leq  s} \lV \bm{A}_{k+1} \bm{x} - \bm{y}_{k+1}\rV_2^2.$$
Then, there exists a universal constant $\alpha_0\in (0,1)$ such that: whenever $\lV \bm{x}_{k}-\bm{x}^{\natural}\rV_2\le \lambda \lV \bm{x}^{\natural}\rV_2$ for some  and some $k\leq K-1$, we have
\begin{align*}
\lV \bm{x}_{k+1}-\bm{x}^{\natural}\rV_2
\le\alpha_0\lV \bm{x}_{k}-\bm{x}^{\natural}\rV_2.
\end{align*}
\end{lemma}
\begin{proof}
Let $k$ be an integer such that $k\leq K$. Define the residual vector $\bm{e}_{k}:=\frac{1}{\sqrt{\beta m}}(\bm{y}_{k+1}-\bm{A}_{k+1}\bm{x}^{\natural})$. Because of  \eqref{event:boundAx}, \cref{bound:Ax-y} implies
\begin{align}\label{eq:estek}
\lV\bm{e}_{k}\rV_2=\frac{1}{\sqrt{\beta m}}\lV\bm{y}_{k+1}-\bm{A}_{k+1}\bm{x}^{\natural}\rV_2
\le  \frac{C_{\lambda}}{\sqrt{\beta}}\lV\bm{x}_k-\bm{x}^{\natural}\rV_2.
\end{align}
Recall that $\{\bm{x}_{k,\ell}\}_{\ell=0}^L$ is the sequence generated by HTP in the $(k+1)$-iteration as stated in Step~7 of \cref{alg:SAM}, and we have set the initial guess $\bm{x}_{k,0}:=\bm{x}_{k}$ and define the output $\bm{x}_{k+1}:=\bm{x}_{k,L}$. By using \eqref{srip:union} with $r =3s$ and applying \cite[Theorem~3.8]{foucart2011hard}, we obtain
%By the special case $J=1$ in \cref{sRIP} we know $\frac{1}{\sqrt{\beta m}}\bm{A}_{k+1}$ satisfies $\gamma$-RIP condition (see \cref{def:RIP}) for any natural number $\gamma\le n$ (as the sparsity level $s$ in \cref{sRIP} can be replaced by $\gamma$). Then by \cite[Theorem~3.8]{foucart2011hard} we have
\begin{align}\label{eq:robustHTP}
\lV\bm{x}_{k,\ell}-\bm{x}^{\natural}\rV_2\le \rho^{\ell}\lV\bm{x}_{k,0}-\bm{x}^{\natural}\rV_2+\tau\frac{1-\rho^\ell}{1-\rho}\lV\bm{e}_{k}\rV_2,
\end{align}
where $\rho=\sqrt{\frac{2\delta^2}{1-\delta^2}}$, and $\tau=\frac{\sqrt{2(1-\delta)}+\sqrt{1+\delta}}{1-\delta}$. Combining \eqref{eq:robustHTP} and \eqref{eq:estek} gives
\begin{align*}
\lV\bm{x}_{k,\ell}-\bm{x}^{\natural}\rV_2
&\le  \rho^{\ell}\lV\bm{x}_k-\bm{x}^{\natural}\rV_2+\frac{\tau(1-\rho^\ell)C_{\lambda}}{\sqrt{\beta}(1-\rho)}\lV\bm{x}_k-\bm{x}^{\natural}\rV_2\\
&=\left(\rho^{\ell}+\frac{\tau(1-\rho^\ell)C_{\lambda}}{\sqrt{\beta}(1-\rho)}\right)\lV \bm{x}_{k}-\bm{x}^{\natural}\rV_2,
\end{align*}
where in the first inequality we used $\bm{x}_{k,0}=\bm{x}_{k}$. Since $\bm{x}_{k+1}=\bm{x}_{k,L}$, we have
\begin{equation*}
\begin{split}
\lV\bm{x}_{k+1}-\bm{x}^{\natural}\rV_2
&\le\left(\rho^{L}+\frac{\tau(1-\rho^L)C_{\lambda}}{\sqrt{\beta}(1-\rho)}\right)\lV \bm{x}_{k}-\bm{x}^{\natural}\rV_2\cr
&\leq\underbrace{\left(\rho^{L}+\frac{2\tau\left(\sqrt{10}\times 10^{-3}+\frac{\lambda}{\sqrt{\beta}}\cdot\sqrt{\frac{21}{20}}\right)}{(1-\lambda)(1-\rho)}\right)}_{\alpha}\lV \bm{x}_{k}-\bm{x}^{\natural}\rV_2,
\end{split}
\end{equation*}
where in the last inequality we have used the expression of $C_{\lambda}$ in \cref{bound:Ax-y} and $\beta\ge \frac{1}{10}$. Therefore, since  $\delta=0.1$, and $\lambda\in[0,\sqrt{\beta}/8]$, it can be verified straightforwardly that for $L=1$, we have
\begin{equation*}
\alpha \leq
\rho+\frac{2\tau\left(\sqrt{10}\times 10^{-3}+\frac18\cdot\sqrt{\frac{21}{20}}\right)}{\frac{7}{8}}\le 0.95,
\end{equation*}
and for $L\ge 2$, we have
\begin{align}\label{def:alpha0}
\alpha
%&\leq\rho^{L}+\frac{2\tau\left(\sqrt{10}\times 10^{-3}+\frac18\cdot\sqrt{\frac{21}{20}}\right)}{\frac{7}{8}(1-\rho)}\nonumber \\
\leq \left.\left(\rho^{L}+\frac{2\tau\left(\sqrt{10}\times10^{-3}+\frac18\cdot\sqrt{\frac{21}{20}}\right)}{\frac{7}{8}(1-\rho)}\right)\right|_{L=2,\delta=0.1}\le 0.7.
\end{align}
Therefore, for all $L\ge1$, we have $\alpha \leq \alpha_0$ where $\alpha_0=0.95$.
\end{proof}

\subsection{Proof of \cref{prop:convergeAlm}}\label{subsec:propconvergeAlm}
%Let the sequences $\{ \bm{y}_k, \bm{A}_{k},  \bm{x}_{k}\}_{k\ge 0}$ is generated by \cref{alg:SAM}.
%Under the events we show that, as long as there is  $\lV \bm{x}_k-\bm{x}^{\natural}\rV_2\le \frac{1}{8}\lV \bm{x}^{\natural}\rV_2$, then for some numerical constant $\zeta_0\in (0,1)$,  the sequence $\{\bm{x}_k\}_{k\ge 1}$ generated by \cref{alg:alm} satisfies
%\begin{align*}
%\lV \bm{x}_{k+1}-\bm{x}^{\natural}\rV_2
%\le\zeta_0\lV \bm{x}_{k}-\bm{x}^{\natural}\rV_2.
%\end{align*}

\begin{proof}%[Proof of \cref{prop:convergeAlm}]
The proposition is proved under the event \eqref{event:boundAx} with $\lambda=\frac18$ and the event \eqref{srip:union} with $K=1,~\beta=1,~r=3s$ and $\delta=0.1$. Without loss of generality and for convinience, we consider only the case $\lV\bm{x}_{0}-\bm{x}^{\natural}\rV_2\le\lV\bm{x}_{0}+\bm{x}^{\natural}\rV_2$ for the given initialization $\bm{x}_{0}$. In this case, the distance is reduced to $\mathrm{dist}\left(\bm{x}_{0},\bm{x}^{\natural}\right)=\lV\bm{x}_{0}-\bm{x}^{\natural}\rV_2$, and we will show that the sequence $\{\|\bm{x}_{k}-\bm{x}^{\sharp}\|_2\}_{k\ge0}$ decreases to $0$ geometrically. In the case of $\lV\bm{x}_{0}-\bm{x}^{\natural}\rV_2>\lV\bm{x}_{0}+\bm{x}^{\natural}\rV_2$, it follows the same proof.

Since $\beta=1$, we have $[m]=\I_1=\I_2=\ldots$, and hence $\bm{A}=\bm{A}_1=\bm{A}_2=\ldots$. Therefore, the event \eqref{srip:union} with  $K=1,~\beta=1,~r=2s,~\delta=0.1$ implies $\|\bm{A}\bm{z}\|_2\geq \sqrt{(1-\delta)m}\|\bm{z}\|_2$ for all $2s$-sparse vector $\bm{z}$. Since $\bm{x}_{k+1}-\bm{x}^{\natural}$ is at most $2s$-sparse for any $k$, we obtain
\begin{equation}\label{eq:RIPAxk+1-xsharp}
\lV  \bm{A}\left(\bm{x}_{k+1}-\bm{x}^{\natural}\right)\rV_2
\ge \sqrt{\left(1-\delta\right)m}\lV\bm{x}_{k+1}-\bm{x}^{\natural}\rV_2,\qquad \forall~k\geq 0.
\end{equation}

Next, we show that, if $\|\bm{x}_k-\bm{x}^{\sharp}\|_2\leq\frac{1}{8}\|\bm{x}^{\sharp}\|_2$, then
\begin{equation}\label{eq:contractexact}
\lV \bm{x}_{k+1}-\bm{x}^{\natural}\rV_2
\le\zeta_0\lV \bm{x}_{k}-\bm{x}^{\natural}\rV_2
\end{equation}
for some universal constant $\zeta_0\in(0,1)$. To this end, we apply the triangle inequality to obtain
\begin{align}\label{eq:triinqaxk+1-yk+1}
\lV  \bm{A}\bm{x}_{k+1}-\bm{y}_{k+1}\rV_2
&=\lV  \bm{A}\bm{x}_{k+1}-\bm{A}\bm{x}^{\natural}+\bm{A}\bm{x}^{\natural}-\bm{y}_{k+1}\rV_2 \nonumber \\
&\ge \lV  \bm{A}\bm{x}_{k+1}-\bm{A}\bm{x}^{\natural}\rV_2-\lV\bm{A}\bm{x}^{\natural}-\bm{y}_{k+1}\rV_2.
\end{align}
Since $\bm{x}_{k+1}=\arg\min_{\|\bm{x}\|_0\leq s}\lV \bm{A}\bm{x}-\bm{y}_{k+1}\rV_2$, it holds that
\begin{equation*}
  \lV  \bm{A}\bm{x}_{k+1}-\bm{y}_{k+1}\rV_2\le\lV\bm{A}\bm{x}^{\natural}-\bm{y}_{k+1}\rV_2.  \end{equation*}
  Plugging it into \eqref{eq:triinqaxk+1-yk+1}, we get
\begin{equation*}
\lV  \bm{A}\bm{x}_{k+1}-\bm{A}\bm{x}^{\natural}\rV_2\le 2\lV\bm{A}\bm{x}^{\natural}-\bm{y}_{k+1}\rV_2\le 2C_{\lambda}\sqrt{m}\lV \bm{x}_{k}-\bm{x}^{\natural}\rV_2,
\end{equation*}
where the last inequality follows from \cref{bound:Ax-y} (which holds true because of the event \eqref{event:boundAx}). By further considering \eqref{eq:RIPAxk+1-xsharp}, we obtain
\begin{align}\label{ineq:onestep}
\lV \bm{x}_{k+1}-\bm{x}^{\natural}\rV_2
\le \frac{2C_{\lambda}}{\sqrt{1-\delta}}\lV \bm{x}_{k}-\bm{x}^{\natural}\rV_2.
\end{align}
Recall that $C_{\lambda}=\frac{2}{1-\lambda}\left(10^{-3}+\lambda \sqrt{\frac{21}{20}}\right)$. Obviously, since $\lambda=\frac18$, and $\delta=0.1$, the factor
$$
\frac{2C_{\lambda}}{\sqrt{1-\delta}}=\frac{2C_{\frac{1}{8}}}{\sqrt{1-0.1}}:=\zeta_0\in(0,1),
$$
which shows \eqref{eq:contractexact}. The numerical value $\zeta_0$ is about 0.6. Since the initial guess satisfies $\|\bm{x}_0-\bm{x}^{\sharp}\|_2\leq\frac18\|\bm{x}^{\sharp}\|_2$,  an induction of \eqref{eq:contractexact} on $k$ implies
\begin{equation}\label{ineq:multistep}
\lV \bm{x}_{k+1}-\bm{x}^{\natural}\rV_2
\le\zeta_0\lV \bm{x}_{k}-\bm{x}^{\natural}\rV_2,\qquad\forall~k.
\end{equation}

Finally, because the constants $\lambda=\frac18$, $\beta=1$, and $\delta=0.1$ are fixed, the probability that both the event \eqref{event:boundAx} with $\lambda=\frac18$ and the event \eqref{srip:union} with $K=1, ~\beta=1,~r=2s,~\delta=0.1$ hold is at least $1-e^{-C'm}$ provided $m\ge C s\log (n/s))$ for universal positive constants $C$ and $C'$. By setting $\zeta=\zeta_0$, we conclude the proof.

%Notice that $\frac{2C(\lambda)}{\sqrt{1-\delta_{2s}}}$ can be small if provided the constants $\lambda,~ \delta_{2s}$ small enough. Thus as long as $\lambda, \delta_{2s}$ sufficiently small, we must have $0<\frac{2C(\lambda)}{\sqrt{1-\delta_{2s}}}<1$. In fact, for constants satisfying
%\begin{align}\label{eg:parameter}
%\lambda\le \frac{1}{8},~ \delta_{2s}\le 0.1,
%\end{align}
%there is $0<\frac{2C(\lambda)}{\sqrt{1-\delta_{2s}}}<1$. Thus we define the constant
%\begin{align*}
%\zeta_0:=\frac{2C(\frac{1}{8})}{\sqrt{1-\delta_{2s}}},
%\end{align*}
% then there is $\zeta_0\in (0,1)$ for $\delta_{2s}\le 0.1$, and

\end{proof}

\subsection{Proof of \cref{localconvergence}}\label{subsec:localconvergence}
%The theorem in proved under event~\eqref{srip:union} and event~\eqref{event:boundAx}.
\begin{proof}%[Proof of \cref{localconvergence}]
The same as \cref{prop:convergeAlm}, without loss of generality and for convinience, we consider only the case $\lV\bm{x}_{0}-\bm{x}^{\natural}\rV_2\le\lV\bm{x}_{0}+\bm{x}^{\natural}\rV_2$ for the given initialization $\bm{x}_{0}$. In this case, the distance is reduced to $\mathrm{dist}\left(\bm{x}_{0},\bm{x}^{\natural}\right)=\lV\bm{x}_{0}-\bm{x}^{\natural}\rV_2$. In the case of $\lV\bm{x}_{0}-\bm{x}^{\natural}\rV_2>\lV\bm{x}_{0}+\bm{x}^{\natural}\rV_2$, it follows the same proof.

We assume the event \eqref{event:boundAx} with $\lambda=\frac{\sqrt{\beta}}{8}$ and the event \eqref{srip:union} with $K$, $r=3s$, $\delta=0.1$. Here $K$ is a positive integer that will be determined later. According to \cref{sRIP} and \cref{bound:Ax}, the probability that these two events hold simultaneously is at least $1-2Ke^{-C_3\beta^2m}$ provided $m\geq C_2\beta^{-2}$, where $C_2, C_3$ are universal positive constants.

With these, Parts (a) and (b) of the theorem are proved respectively as in the following.

\paragraph{(a)} This part is a direct consequence of \cref{convergence:exact}. Suppose $\lV\bm{x}_{k}-\bm{x}^{\natural}\rV_2\leq \frac{\sqrt{\beta}}{8} \lV\bm{x}^{\natural}\rV_2$. Under the two events \eqref{srip:union}  and \eqref{event:boundAx},  \cref{convergence:exact} implies
$$
\lV\bm{x}_{k+1}-\bm{x}^{\natural}\rV_2\leq \alpha_0\lV\bm{x}_{k}-\bm{x}^{\natural}\rV_2\leq \frac{\sqrt{\beta}}{8} \lV\bm{x}^{\natural}\rV_2.
$$
This by induction implies that: whenever the initialization satisfies $\lV\bm{x}_0-\bm{x}^{\natural}\rV_2\leq \frac{\sqrt{\beta}}{8} \lV\bm{x}^{\natural}\rV_2$, we have always
$$
\lV\bm{x}_{k+1}-\bm{x}^{\natural}\rV_2\leq \alpha_0\lV\bm{x}_{k}-\bm{x}^{\natural}\rV_2,\quad\forall~0\le k\le K-1.
$$

\paragraph{(b)}
Let
$$\E_{k}=\left\{i:\mathrm{sgn}\left(\bm{a}_i^T \bm{x}_{k}\right)\neq\mathrm{sgn}\left(\bm{a}_i^T \bm{x}^{\natural}\right)\right\},$$
and define $\D_{k}:=\I_{k+1}\bigcap\E_{k}$. Then
\begin{align}\label{inqR:y-Ax}
\begin{split}
\lv\l \bm{y}_{k+1}-\bm{A}_{k+1}\bm{x}^{\natural}, \bm{A}_{k+1}\bm{x}^{\natural}\r\rv&=
\lv\l \bm{y}_{\I_{k+1}}\odot \mathrm{sgn}(\bm{A}_{k+1}\bm{x}_{k})-\bm{A}_{k+1}\bm{x}^{\natural}, \bm{A}_{k+1}\bm{x}^{\natural}\r\rv\\
&=\Big|\sum_{i\in \I_{k+1}}\big( \lv\bm{a}_i^T \bm{x}^{\natural}\rv\cdot\mathrm{sgn}\left(\bm{a}_i^T \bm{x}_{k}\right)-\bm{a}_i^T \bm{x}^{\natural}\big)(\bm{a}_i^T \bm{x}^{\natural})\Big|\\
&=\Big|\sum_{i\in \I_{k+1}} \lv\bm{a}_i^T \bm{x}^{\natural}\rv^2\big(\mathrm{sgn}\left(\bm{a}_i^T \bm{x}^{\sharp}\right)\mathrm{sgn}\left(\bm{a}_i^T \bm{x}_{k}\right)-1\big)\Big|\\
&=2\sum_{i\in \D_{k}}\lv\bm{a}_i^T \bm{x}^{\natural} \rv^2=2\sum_{i\in \D_{k}}y_i^2\ge 2 \lv \D_{k}\rv y_{\min}^2,
\end{split}
\end{align}
where $y_{\min}$ is the minimum nonzero element in $\left\{{y}_i\right\}_{i=1}^m$. On the other hand, because of  \eqref{srip:union} and \eqref{event:boundAx}, for any $k\in\{1,2,\ldots,K\}$,
\begin{align}\label{inqL:y-Ax}
\begin{split}
\lv\l \bm{y}_{k+1}-\bm{A}_{k+1}\bm{x}^{\natural}, \bm{A}_{k+1}\bm{x}^{\natural}\r\rv
&\le \lV \bm{y}_{k+1}-\bm{A}_{k+1}\bm{x}^{\natural}\rV_2\cdot\lV\bm{A}_{k+1}\bm{x}^{\natural} \rV_2\\
&\le C_{\lambda}\sqrt{m}\lV \bm{x}_{k}-\bm{x}^{\natural}\rV_2\cdot \sqrt{(1+\delta)\beta m}\lV\bm{x}^{\natural}\rV_2\\
&\le \alpha_0^k mC_{\lambda}\sqrt{\beta(1+\delta)}\lV \bm{x}_{0}-\bm{x}^{\natural}\rV_2\lV\bm{x}^{\natural}\rV_2\\
&\le  \alpha_0^k \lambda mC_{\lambda}\sqrt{\beta(1+\delta)}\lV\bm{x}^{\natural}\rV_2^2,
\end{split}
\end{align}
where the second line follows from \cref{bound:Ax-y} and \eqref{srip:union}, the third line follows from \cref{convergence:exact} (as $L\ge 2s$) and the initial guess satisfies $\lV \bm{x}_{0}-\bm{x}^{\natural}\rV_2\le \frac{\sqrt{\beta}}{8}\lV\bm{x}^{\natural}\rV_2$ with $\lambda\in [0,\frac{\sqrt{\beta}}{8}]$, and the last line follows from the assumption $\lV \bm{x}_{0}-\bm{x}^{\natural}\rV_2\le \lambda\lV\bm{x}^{\natural}\rV_2$. Combining \eqref{inqR:y-Ax} and \eqref{inqL:y-Ax} gives
\begin{align*}
\lv \D_{k}\rv y_{\min}^2\le \frac{1}{2}\alpha_0^k \lambda mC_{\lambda}\sqrt{\beta(1+\delta_s)}\lV\bm{x}^{\natural}\rV_2^2.
\end{align*}
Choosing $K$ to be the minimum integer such that
\begin{align}\label{cond:finite-k}
\frac{1}{2}\alpha_0^{K-1} \lambda mC_{\lambda}\sqrt{\beta(1+\delta_s)}\lV\bm{x}^{\natural}\rV_2^2<y_{\min}^2
\leq \frac{1}{2}\alpha_0^{K-2} \lambda mC_{\lambda}\sqrt{\beta(1+\delta_s)}\lV\bm{x}^{\natural}\rV_2^2,
\end{align}
it holds that $\lv \D_{k}\rv<1$ for all $k\geq K-1$. Since $|\D_{k}|$ is a nonnegative integer, one has $\lv \D_{k}\rv=0$ for all $k\geq K-1$.

Let us estimate $K$ satisfying \eqref{cond:finite-k}. Notice that $a_{ij}\sim \N(0,1)$, and $\{\bm{a}_i^T\bm{x}^{\natural}\}_{i=1}^m$ are independent. By the proof of \cite[Theorem~1]{CAI2022367}, we have
$$
\mathsf{P}\left(\lv\bm{a}_i^T\bm{x}^{\natural} \rv\ge m^{-2}\sqrt{\frac{\pi}{2}}\lV \bm{x}^{\natural}\rV_2,\quad \forall i\in [m] \right) \ge 1-\frac{1}{m}.
$$
Since $y_i=\lv \bm{a}_i^T\bm{x}^{\natural}\rv $ and $i\in [m]$, the above inequality implies
\begin{align}\label{lowbound:y}
\mathsf{P}\left(y_{\min}\ge m^{-2}\sqrt{\frac{\pi}{2}}\lV \bm{x}^{\natural}\rV_2\right) \ge 1-\frac{1}{m}.
\end{align}
Plugging it into \eqref{cond:finite-k}, we obtain that, with probability at least $1-\frac{1}{m}$,
\begin{equation}\label{eq:event3}
\frac{1}{2}\alpha_0^{K-2} \lambda mC_{\lambda}\sqrt{\beta(1+\delta)}\lV\bm{x}^{\natural}\rV_2^2
\geq m^{-4}\frac{\pi}{2}\lV \bm{x}^{\natural}\rV_2^2,
\end{equation}
which is equivalent to, by noticing $\lambda\in [0, \frac{\sqrt{\beta}}{8}]$ and  $C_{\frac{\sqrt{\beta}}{8}}<1$,
\begin{equation*}
\begin{split}
K&\leq \frac{5\log m+\log\Big(\lambda C_{\lambda}\sqrt{\beta(1+\delta)}/\pi\Big)}{\log \alpha_0^{-1}} +2
\leq \frac{5\log m+\log \left(\frac{\sqrt{\beta}}{8\pi} C_{\frac{\sqrt{\beta}}{8}}\sqrt{\beta(1+\delta)}\right) }{\log \alpha_0^{-1}} +2 \cr
%\leq \left(\left\lfloor\frac{\log \Big(\frac{\sqrt{\beta}}{8\pi} C_{\frac{\sqrt{\beta}}{8}}\sqrt{\beta(1+\delta)}\Big)}{\log \alpha_0^{-1}}\right\rfloor+7\right)\cdot\log m\cr
&\leq \underbrace{\left(\frac{5}{\log \alpha_0^{-1}}+2\right)}_{C_1}\cdot\log m.
\end{split}
\end{equation*}
Since $\delta=0.1$ and $L\ge 2s\ge 2$,  by \eqref{def:alpha0} we then know $\alpha_0 \le 0.7$, and the numerical value of $C_1$ is about $16$. In summary, we have
$$
\lv \D_{k}\rv=0\ \text{for all}\ k\geq K-1 \text{~with some~}\  K\leq C_1\log m.
$$

By the definition of $\D_{k}$, we have
\begin{align*}
\bm{y}_{K}= \mathrm{sgn}(\bm{A}_{K}\bm{x}_{K-1})\odot \bm{y}_{\I_{K}}=\mathrm{sgn}(\bm{A}_{K}\bm{x}^{\natural})\odot \lvert \bm{A}_{K}\bm{x}^{\natural}\rvert=\bm{A}_{K}\bm{x}^{\natural}.
\end{align*}
Therefore, in the $K$-th iteration of \cref{alg:SAM}, we are solving the following problem
\begin{align*}
\bm{x}_{K}=\mathop{\mathrm{arg~min}}\limits_{ \lV \bm{x}\rV_0\leq  s}\lV \bm{A}_{K} \bm{x} - \bm{A}_{K}\bm{x}^{\natural}\rV_2^2,
\end{align*}
via HTP (\cref{alg:htp}), and the maximum allowed iteration number $L$ of HTP satisfies $L\ge 2s$. Furthermore, in event \eqref{srip:union}, the coefficient matrix $\bm{A}_K$ satisfies RIP for $3s$-sparse vectors with constant $\delta=0.1\leq\frac13$. Altogether, according to the exact recovery result of HTP stated in \cite[Theorem 5]{bouchot2016hard}, $\bm{x}_{K}=\bm{x}_{\sharp}$, which obviously implies
\begin{equation}\label{eq:exact}
\bm{x}_{k}=\bm{x}_{\sharp}, \qquad\forall~k\geq K.
\end{equation}

Finally, in the above proof of \eqref{eq:exact}, besides events \eqref{srip:union} and \eqref{event:boundAx}, we have also assumed event \eqref{eq:event3}. By a simple union bound, we obtain that the probability for \eqref{eq:exact} is at least $1-2Ke^{-C_3\beta^2m}-m^{-1}$ provided $m\geq C_2\beta^{-2}s\log(n/s)$.

\end{proof}

\subsection{Proof of \cref{converge:alm}}\label{subsec:proofalm}
\begin{proof}
The proof is almost the same as that of Part (b) of \cref{localconvergence}. The only difference is that, when $\beta=1$, the set $\I_k$ satisfies $[m]=\I_1=\I_2=\ldots$. As a consequence, we have $\bm{A}=\bm{A}_1=\bm{A}_2=\ldots$. Therefore, the simultaneous RIP \eqref{srip:union} holds for $K=+\infty$ with probability at least $1-2e^{-c_2m}$. Thus, the probability in the theorem statement is $1-e^{-C_6m}-m^{-1}$.
\end{proof}

%\begin{proof}
%By \cref{ARIP} and \cref{bound:Ax}, using same argument to the proof of \cref{bound:Ax-y}, we know there exists constants $c_8,c_9$, as long as $m\geq c_8 s\log(n/s)$, then \eqref{bound:sign} holds with $\beta=1$ for all $k\in\mathbb{N}_+$ with probability at least $1-3e^{-c_9 m}$. So using the same argument to the proof of \cref{localconvergence} with $\beta=1$, the statements in \cref{converge:alm} holds for all $k\in\mathbb{N}_+$.
%\end{proof}

\subsection{Supporting lemmas}\label{subsec:supportlemmas}
In this subsection, we present some supporting lemmas from the literature, to make the paper more self-contained.

The following \cref{ARIP} is well known in compressed sensing theory \cite{foucart2013invitation,candes2005decoding}, which states that the random Gaussian matrix $\frac{1}{\sqrt{m}}\bm{A}$ satisfies the RIP as long as $m$ is sufficiently large.
\begin{lemma}[{\cite[Theorem 9.27]{foucart2013invitation}}]\label{ARIP}
Let each entry of $\bm{A}$ be independently sampled from Gaussian $\N(0,1)$. There exists some universal positive constants $\tilde{c}_1,\tilde{c}_2$ such that: For any natural number $r \leq n$ and any $\delta_r\in(0,1)$, with probability at least $1-e^{-\tilde{c}_1 m}$, $\frac{1}{\sqrt{m}}\bm{A}$ satisfies $r$-RIP with constant $\overline{\delta}_r$, i.e.,
\begin{equation*}
\left(1-\overline{\delta}_r\right)\lV\bm{x}\rV_2^2\le \frac{1}{m}\lV \bm{A}\bm{x}\rV_2^2\le \left(1+\overline{\delta}_r\right)\lV\bm{x}\rV_2^2, \qquad\forall~\lV\bm{x}\rV_0\le r,
\end{equation*}
provided $m\ge \tilde{c}_2\delta_r^{-2}r\log\left(n/r\right)$.
\end{lemma}
%\begin{lemma}\label{bound:Ax}
%(\cite[Lemma 25]{soltanolkotabi2019structured}) For any $s$-sparse signal $\bm{x}^{\natural}$, assume the sampling vectors $\left\{\bm{a}_i\right\}_{i=1}^{m}$ are i.i.d. Gaussian random vectors distributed as $\N(\bm{0},\bm{I}_n)$.  For $\lambda\in(0, \frac{1}{8})$, fixing any $\varepsilon_0>0$, there exist numerical constants $\tilde{c}_3, \tilde{c}_4$, as long as the sample size $m$ satisfies
%$$m>\tilde{c}_3 s\log\left(n/s\right),$$
%then with probability at least $1-e^{-\tilde{c}_4 m}$, it holds
%\begin{align*}
%\frac{1}{m}\mathop{\sum}\limits_{i=1}^{m}\lv \bm{a}_i^T\bm{x}^{\natural}\rv^2\cdot \bm{1}_{\left\{\left(\bm{a}_i^T\bm{x}\right)\left(\bm{a}_i^T\bm{x}^{\natural}\right)\le 0\right\}}
%\le\frac{1}{\left(1-\lambda\right)^2}\left(\varepsilon_0+\lambda \sqrt{\frac{21}{20}}\right)\lV\bm{x}-\bm{x}^{\natural}\rV_2^2
%\end{align*}
%for all at most $s$-sparse vector $\bm{x}\in\mathbb{R}^n$ satisfying $\mathrm{dist}\left(\bm{x},\bm{x}^{\natural}\right)\le \lambda \lV\bm{x}^{\natural}\rV_2$.
%\end{lemma}
%The above \cref{bound:Ax} is a direct modification of \cite[Lemma 25]{soltanolkotabi2019structured}. And same  modification can also be found in \cite[Lemma C.1]{jagatap2019sample}.
%\begin{lemma}(Hoeffding's inequality)\label{Hoeffding} Let $X_1, X_2,\cdots, X_m$ be independent mean-zero random
%variables satisfying $\lv X_i\rv\le B_i,\ i\in[m]$ almost surely. Then, for all $u>0$,
%\begin{align*}
%\mathbb{P}\left(\lv\sum_{i=1}^m X_i\rv \ge u\right)\le 2\exp\left( -\frac{u^2}{2\sum_{i=1}^m B_i^2}\right).
%\end{align*}
%\end{lemma}
\begin{lemma}[Bernstein's inequality,{\cite[Corollary 7.32]{foucart2013invitation}}]\label{Bernstein} Let $X_1, X_2,\cdots, X_m$ be independent mean-zero subexponential random
variables, i.e., $\mathsf{P}\left(\lv X_i\rv \ge u\right)\le \tilde{c}_3 e^{-\tilde{c}_4 u}$ for some constants $\tilde{c}_3,\tilde{c}_4>0$ for all $u>0$, $i\in [m]$. Then it holds
\begin{align*}
\mathsf{P}\left(\lv\sum_{i=1}^m X_i\rv \ge u\right)\le 2\exp\left( -\frac{(\tilde{c}_4 u)^2/2}{2\tilde{c}_3 m+\tilde{c}_4u}\right).
\end{align*}
\end{lemma}

Let $\bm{B}\in \mathbb{R}^{m\times n}$ be a random matrix with $\mathsf{E}\left(\lV\bm{B}\bm{x}\rV_2^2\right)=\lV\bm{x}\rV_2^2$ for any $\bm{x}\in\mathbb{R}^n$. Then, for any $\bm{x}\in\mathbb{R}^n$, the random variable $\lV\bm{B}\bm{x}\rV_2^2$ is said to be strongly concentrated about its expected value if
\begin{align}\label{strongconcentration}
\mathsf{P}\left(\lv \lV\bm{B}\bm{x}\rV_2^2-\lV\bm{x}\rV_2^2\rv\ge \tilde{\varepsilon} \lV\bm{x}\rV_2^2 \right)\le 2e^{-\tilde{c}(\tilde{\varepsilon})m}, \qquad 0<\tilde{\varepsilon}<1,
\end{align}
where %the probability is taken over all $m\times n$ matrices $\bm{B}$, and
$\tilde{c}(\tilde{\varepsilon})$ is a positive constant depending only on $\tilde{\varepsilon}$ for any $\tilde{\varepsilon}\in (0,1)$.
\begin{lemma}[{\cite[Lemma 5.1]{Baraniuk2008A}}]\label{lemma:concentration} Let $\bm{B}\in \mathbb{R}^{m\times n}$ be a random matrix that satisfies the concentration inequality \eqref{strongconcentration}. Then for any $\tilde{\delta}\in(0,1)$ and any $\S\subseteq [n]$ with $|\S|=r$, it holds
$$
(1-\tilde{\delta})\lV\bm{x}\rV_2\le \lV\bm{B}\bm{x}\rV_2\le (1+\tilde{\delta})\lV\bm{x}\rV_2, \qquad\forall~ \bm{x}~\text{satisfies}~ \mathrm{support}(\bm{x})\subseteq \S
$$
with probability at least $1-2(12/\tilde{\delta})^r e^{-\tilde{c}(\tilde{\delta}/2)^2 m}$.
\end{lemma}
\section{Conclusion}\label{section:conclusion}
We have proposed a novel stochastic method named SAM for sparse phase retrieval problem, which is based on a alternating minimization framework. It has been verified that the proposed SAM finds the exact solution in few number of iterations in our theory and experiments.  Moreover, numerical experiments also show that our algorithm SAM outperforms the comparative algorithms such as ThWF, SPARTA, CoPRAM and standard alternating minimization without randomness in terms of sample efficiency .

%However, in the experiments, it shows that the stochastic version alternating minimization has better performance than standard alternating minimization in terms of sample efficiency. Then how to give some reasonable explanations for such improvement? It still remains interesting to find out the reason with rigorous analysis.

\section*{Acknowledgment}
The work of J.-F. Cai is partially supported by Hong Kong Research Grants Council (HKRGC) GRF grants 16309518, 16309219, 16310620, and 16306821. Y. Jiao was supported in part by the National Science Foundation of China under Grant 11871474 and by the research fund of KLATASDSMOE, and by  the Natural Science Foundation of Hubei Province (No. 2019CFA007). X.-L. Lu is partially supported by the National Key Research and Development Program of China (No.~2020YFA0714200), the National Science Foundation of China (No.~11871385) and the Natural Science Foundation of Hubei Province (No.~2019CFA007).

\bibliographystyle{abbrv}
\bibliography{sparsePR}

\begin{thebibliography}{10}

\bibitem{bahmani2017flexible}
S.~Bahmani, J.~Romberg, et~al.
\newblock A flexible convex relaxation for phase retrieval.
\newblock {\em Electronic Journal of Statistics}, 11(2):5254--5281, 2017.

\bibitem{balan2006signal}
R.~Balan, P.~Casazza, and D.~Edidin.
\newblock On signal reconstruction without phase.
\newblock {\em Applied and Computational Harmonic Analysis}, 20(3):345--356,
  2006.

\bibitem{Baraniuk2008A}
R.~Baraniuk, M.~Davenport, R.~Devore, and M.~Wakin.
\newblock A simple proof of the restricted isometry property for random
  matrices.
\newblock {\em Constructive Approximation}, 28(3):253--263, 2008.

\bibitem{blumensath2008iterative}
T.~Blumensath and M.~E. Davies.
\newblock Iterative thresholding for sparse approximations.
\newblock {\em Journal of Fourier Analysis and Applications}, 14(5-6):629--654,
  2008.

\bibitem{bouchot2016hard}
J.-L. Bouchot, S.~Foucart, and P.~Hitczenko.
\newblock Hard thresholding pursuit algorithms: number of iterations.
\newblock {\em Applied and Computational Harmonic Analysis}, 41(2):412--435,
  2016.

\bibitem{cai2021solving}
J.~Cai, M.~Huang, D.~Li, and Y.~Wang.
\newblock Solving phase retrieval with random initial guess is nearly as good
  as by spectral initialization.
\newblock {\em arXiv preprint arXiv:2101.03540}, 2021.

\bibitem{CAI2022367}
J.-F. Cai, J.~Li, X.~Lu, and J.~You.
\newblock Sparse signal recovery from phaseless measurements via hard
  thresholding pursuit.
\newblock {\em Applied and Computational Harmonic Analysis}, 56:367--390, 2022.

\bibitem{cai2018solving}
J.-F. Cai and K.~Wei.
\newblock Solving systems of phaseless equations via {R}iemannian optimization
  with optimal sampling complexity.
\newblock {\em arXiv preprint arXiv:1809.02773}, 2018.

\bibitem{cai2016optimal}
T.~T. Cai, X.~Li, and Z.~Ma.
\newblock Optimal rates of convergence for noisy sparse phase retrieval via
  thresholded wirtinger flow.
\newblock {\em The Annals of Statistics}, 44(5):2221--2251, 2016.

\bibitem{candes2015phase}
E.~J. Candes, Y.~C. Eldar, T.~Strohmer, and V.~Voroninski.
\newblock Phase retrieval via matrix completion.
\newblock {\em SIAM review}, 57(2):225--251, 2015.

\bibitem{candes2015phase1}
E.~J. Candes, X.~Li, and M.~Soltanolkotabi.
\newblock Phase retrieval via wirtinger flow: Theory and algorithms.
\newblock {\em IEEE Transactions on Information Theory}, 61(4):1985--2007,
  2015.

\bibitem{candes2005decoding}
E.~J. Candes and T.~Tao.
\newblock Decoding by linear programming.
\newblock {\em IEEE transactions on Information Theory}, 51(12):4203--4215,
  2005.

\bibitem{chen2015solving}
Y.~Chen and E.~Candes.
\newblock Solving random quadratic systems of equations is nearly as easy as
  solving linear systems.
\newblock In {\em Advances in Neural Information Processing Systems}, pages
  739--747, 2015.

\bibitem{chen2019gradient}
Y.~Chen, Y.~Chi, J.~Fan, and C.~Ma.
\newblock Gradient descent with random initialization: Fast global convergence
  for nonconvex phase retrieval.
\newblock {\em Mathematical Programming}, 176(1-2):5--37, 2019.

\bibitem{fan2014primal}
Q.~Fan, Y.~Jiao, and X.~Lu.
\newblock A primal dual active set algorithm with continuation for compressed
  sensing.
\newblock {\em IEEE Transactions on Signal Processing}, 62(23):6276--6285,
  2014.

\bibitem{fienup1982phase}
J.~R. Fienup.
\newblock Phase retrieval algorithms: a comparison.
\newblock {\em Applied Optics}, 21(15):2758--2769, 1982.

\bibitem{foucart2011hard}
S.~Foucart.
\newblock Hard thresholding pursuit: an algorithm for compressive sensing.
\newblock {\em SIAM Journal on Numerical Analysis}, 49(6):2543--2563, 2011.

\bibitem{foucart2013invitation}
S.~Foucart and H.~Rauhut.
\newblock {\em A Mathematical Introduction to Compressive Sensing.}
\newblock Applied and Numerical Harmonic Analysis. Birkhäuser, 2013.

\bibitem{gao2016gauss}
B.~Gao and Z.~Xu.
\newblock Gauss-newton method for phase retrieval.
\newblock {\em IEEE Transactions on Signal Processing}, 65(22):5885--5896,
  2017.

\bibitem{gerchberg1972practical}
R.~W. Gerchberg.
\newblock A practical algorithm for the determination of the phase from image
  and diffraction plane pictures.
\newblock {\em Optik}, 35:237--246, 1972.

\bibitem{goldstein2016phasemax}
T.~Goldstein and C.~Studer.
\newblock Phasemax: Convex phase retrieval via basis pursuit.
\newblock {\em IEEE Transactions on Information Theory}, 64(4):2675--2689,
  2018.

\bibitem{hand2018elementary}
P.~Hand and V.~Voroninski.
\newblock An elementary proof of convex phase retrieval in the natural
  parameter space via the linear program phasemax.
\newblock {\em Communications in Mathematical Sciences}, 16(7):2047--2051,
  2018.

\bibitem{harrison1993phase}
R.~W. Harrison.
\newblock Phase problem in crystallography.
\newblock {\em JOSA a}, 10(5):1046--1055, 1993.

\bibitem{jagatap2019sample}
G.~Jagatap and C.~Hegde.
\newblock Sample-efficient algorithms for recovering structured signals from
  magnitude-only measurements.
\newblock {\em IEEE Transactions on Information Theory}, 65(7):4434--4456,
  2019.

\bibitem{li2013sparse}
X.~Li and V.~Voroninski.
\newblock Sparse signal recovery from quadratic measurements via convex
  programming.
\newblock {\em SIAM Journal on Mathematical Analysis}, 45(5):3019--3033, 2013.

\bibitem{8918236}
Z.~{Li}, J.~{Cai}, and K.~{Wei}.
\newblock Toward the optimal construction of a loss function without spurious
  local minima for solving quadratic equations.
\newblock {\em IEEE Transactions on Information Theory}, 66(5):3242--3260,
  2020.

\bibitem{ma2018globally}
C.~Ma, X.~Liu, and Z.~Wen.
\newblock Globally convergent levenberg-marquardt method for phase retrieval.
\newblock {\em IEEE Transactions on Information Theory}, 65(4):2343--2359,
  2018.

\bibitem{miao1999extending}
J.~Miao, P.~Charalambous, J.~Kirz, and D.~Sayre.
\newblock Extending the methodology of x-ray crystallography to allow imaging
  of micrometre-sized non-crystalline specimens.
\newblock {\em Nature}, 400(6742):342--344, 1999.

\bibitem{miao2008extending}
J.~Miao, T.~Ishikawa, Q.~Shen, and T.~Earnest.
\newblock Extending x-ray crystallography to allow the imaging of
  noncrystalline materials, cells, and single protein complexes.
\newblock {\em Annu. Rev. Phys. Chem.}, 59:387--410, 2008.

\bibitem{NeedellCoSaMP}
D.~Needell and J.~A. Tropp.
\newblock Cosamp: Iterative signal recovery from incomplete and inaccurate
  samples.
\newblock {\em Applied and Computational Harmonic Analysis}, 26(3):301--321,
  2009.

\bibitem{netrapalli2013phase}
P.~Netrapalli, P.~Jain, and S.~Sanghavi.
\newblock Phase retrieval using alternating minimization.
\newblock In {\em Advances in Neural Information Processing Systems}, pages
  2796--2804, 2013.

\bibitem{shechtman2015phase}
Y.~Shechtman, Y.~C. Eldar, O.~Cohen, H.~N. Chapman, J.~Miao, and M.~Segev.
\newblock Phase retrieval with application to optical imaging: a contemporary
  overview.
\newblock {\em IEEE Signal Processing Magazine}, 32(3):87--109, 2015.

\bibitem{soltanolkotabi2019structured}
M.~Soltanolkotabi.
\newblock Structured signal recovery from quadratic measurements: Breaking
  sample complexity barriers via nonconvex optimization.
\newblock {\em IEEE Transactions on Information Theory}, 65(4):2374--2400,
  2019.

\bibitem{sun2018geometric}
J.~Sun, Q.~Qu, and J.~Wright.
\newblock A geometric analysis of phase retrieval.
\newblock {\em Foundations of Computational Mathematics}, 18(5):1131--1198,
  2018.

\bibitem{tan2019online}
Y.~S. Tan and R.~Vershynin.
\newblock Online stochastic gradient descent with arbitrary initialization
  solves non-smooth, non-convex phase retrieval.
\newblock {\em arXiv preprint arXiv:1910.12837}, 2019.

\bibitem{tan2019phase}
Y.~S. Tan and R.~Vershynin.
\newblock Phase retrieval via randomized kaczmarz: Theoretical guarantees.
\newblock {\em Information and Inference: A Journal of the IMA}, 8(1):97--123,
  2019.

\bibitem{WaldspurgerPhase}
I.~Waldspurger.
\newblock Phase retrieval with random gaussian sensing vectors by alternating
  projections.
\newblock {\em IEEE Transactions on Information Theory}, 64(5):3301--3312,
  2018.

\bibitem{waldspurger2015phase}
I.~Waldspurger, A.~d'Aspremont, and S.~Mallat.
\newblock Phase recovery, maxcut and complex semidefinite programming.
\newblock {\em Mathematical Programming}, 149(1):47--81, 2015.

\bibitem{walther1963question}
A.~Walther.
\newblock The question of phase retrieval in optics.
\newblock {\em Journal of Modern Optics}, 10(1):41--49, 1963.

\bibitem{wang2016solving}
G.~Wang and G.~Giannakis.
\newblock Solving random systems of quadratic equations via truncated
  generalized gradient flow.
\newblock In {\em Advances in Neural Information Processing Systems}, pages
  568--576, 2016.

\bibitem{wang2017scalable}
G.~Wang, G.~B. Giannakis, and J.~Chen.
\newblock Scalable solvers of random quadratic equations via stochastic
  truncated amplitude flow.
\newblock {\em IEEE Transactions on Signal Processing}, 65(8):1961--1974, 2017.

\bibitem{wang2017solving}
G.~Wang, G.~B. Giannakis, and Y.~C. Eldar.
\newblock Solving systems of random quadratic equations via truncated amplitude
  flow.
\newblock {\em IEEE Transactions on Information Theory}, 64(2):773--794, 2017.

\bibitem{wang2016sparse}
G.~Wang, L.~Zhang, G.~B. Giannakis, M.~Ak{\c{c}}akaya, and J.~Chen.
\newblock Sparse phase retrieval via truncated amplitude flow.
\newblock {\em IEEE Transactions on Signal Processing}, 66(2):479--491, 2018.

\bibitem{wang2014phase}
Y.~Wang and Z.~Xu.
\newblock Phase retrieval for sparse signals.
\newblock {\em Applied and Computational Harmonic Analysis}, 37(3):531--544,
  2014.

\bibitem{wei2015solving}
K.~Wei.
\newblock Solving systems of phaseless equations via kaczmarz methods: A proof
  of concept study.
\newblock {\em Inverse Problems}, 31(12):125008, 2015.

\bibitem{wu2021hadamard}
F.~Wu and P.~Rebeschini.
\newblock Hadamard wirtinger flow for sparse phase retrieval.
\newblock In {\em International Conference on Artificial Intelligence and
  Statistics}, pages 982--990. PMLR, 2021.

\bibitem{zhang2016reshaped}
H.~Zhang and Y.~Liang.
\newblock Reshaped wirtinger flow for solving quadratic system of equations.
\newblock In {\em Advances in Neural Information Processing Systems}, pages
  2622--2630, 2016.

\end{thebibliography}
\end{document}